\documentclass[12pt]{amsart}
\usepackage{amsfonts, amssymb, latexsym, graphicx, hyperref, amsmath}
\usepackage{tikz}
\usepackage{tikz-cd}
\usetikzlibrary{calc, shapes, backgrounds,arrows,positioning,decorations.pathmorphing}

\newtheorem{theorem}{Theorem}[section]
\newtheorem{proposition}[theorem]{Proposition}
\newtheorem{observation}[theorem]{Observation}
\newtheorem{lemma}[theorem]{Lemma}

\newtheorem{claim}[theorem]{Claim}
\newtheorem*{claim*}{Claim}
\newtheorem{corollary}[theorem]{Corollary}
\newtheorem{Main Conjecture}[theorem]{Main Conjecture}
\newtheorem{conjecture}[theorem]{Conjecture}
\newtheorem{problem}[theorem]{Problem}
\theoremstyle{remark}

\newtheorem{example}[theorem]{Example}

\theoremstyle{plain}

\newcommand{\cellsize}{14}
\newlength{\cellsz} 
\newsavebox{\cell}
\sbox{\cell}{\begin{picture}(\cellsize,\cellsize)
\put(0,0){\line(1,0){\cellsize}}
\put(0,0){\line(0,1){\cellsize}}
\put(\cellsize,0){\line(0,1){\cellsize}}
\put(0,\cellsize){\line(1,0){\cellsize}}
\end{picture}}
\newcommand\cellify[1]{\def\thearg{#1}\def\nothing{}%
\ifx\thearg\nothing
\vrule width0pt height\cellsz depth0pt\else
\hbox to 0pt{\usebox{\cell} \hss}\fi%
\vbox to \cellsz{
\vss
\hbox to \cellsz{\hss$#1$\hss}
\vss}}
\newcommand\tableau[1]{\vtop{\let\\\cr
\baselineskip -16000pt \lineskiplimit 16000pt \lineskip 0pt
\ialign{&\cellify{##}\cr#1\crcr}}}

\newcommand{\excise}[1]{}

\begin{document}
\pagestyle{plain}
\title{Reduced word enumeration, complexity, and randomization}
\author{Cara Monical}
\address{Sandia National Laboratories, Albuquerque, NM 87185, USA}
\email{caramonical.math@gmail.com}
\author{Benjamin Pankow}
\author{Alexander Yong}
\address{Dept.~of Mathematics, U.~Illinois at Urbana-Champaign, Urbana, IL 61801, USA} 
\email{bpankow2@illinois.edu, ayong@illinois.edu}
\date{February 15, 2019}

\begin{abstract}
A \emph{reduced word} of a permutation $w$ is a minimal length expression of $w$ 
as a product of simple transpositions. We examine the computational complexity, formulas and (randomized) algorithms for their enumeration. 
In particular, we prove that the \emph{Edelman-Greene statistic}, defined by S.~Billey-B.~Pawlowski, is typically exponentially large. This implies a result
of B.~Pawlowski, that it has exponentially growing expectation. Our result is established by a formal run-time analysis of A.~Lascoux-M.-P.~Sch\"utzenberger's \emph{transition algorithm}. 
The more general problem of
Hecke word enumeration, and its closely related question of
counting \emph{set-valued standard Young tableaux}, is also investigated.
The latter enumeration problem is further motivated by work on \emph{Brill-Noether varieties} due to M.~Chan-N.~Pflueger and D.~Anderson-L.~Chen-N.~Tarasca.
\end{abstract}

\maketitle

\vspace{-.2in}
\section{Introduction}\label{sec:1}

\subsection{Reduced word combinatorics} Let $S_n$ denote the symmetric group on $\{1,2,\ldots,n\}$.
Each $w\in S_n$
can be expressed as a product of $\ell(w)$ simple transpositions $s_i= (i, i+1)$, where $\ell(w)$ is the number of \emph{inversions} of $w$, 
i.e., pairs $i<j$ such that $w(i)>w(j)$. Such an expression $w=s_{i_1}s_{i_2}\cdots s_{i_{\ell(w)}}$
is a \emph{reduced word} for $w$. 

Let ${\sf Red}(w)$ be the set of reduced words for $w$.
R.~P.~Stanley \cite{Stanley1984} defined a symmetric function 
$F_w$ such that 
\begin{equation}
\label{eqn:firstformula}
\#{\sf Red}(w)=
\text{the coefficient of
$x_1 x_2\cdots x_{\ell(w)}$ in $F_w$.}
\end{equation}
In connection to \emph{ibid.},
P.~Edelman-C.~Greene \cite[Section~8]{Edelman.Greene} proved that
\begin{equation}
\label{eqn:stanleyformula}
\#{\sf Red}(w)=\sum_{\lambda} a_{w,\lambda}f^{\lambda}, \text{ \ where}
\end{equation}
\begin{itemize}
\item $f^{\lambda}$ is the number of \emph{standard Young tableaux}
of shape $\lambda$, that is, row and column increasing bijective fillings of the Young diagram of $\lambda$ using $1,2,\ldots,|\lambda|$. The  
\emph{hook-length formula} of J.~S.~Frame-G.~de B.~Robinson-R.~M.~Thrall \cite{hook}
states
\begin{equation}
\label{eqn:HLF}
f^{\lambda}=\frac{|\lambda|!}{\prod_b h_b},
\end{equation}
where the product is over all boxes $b\in \lambda$ and $h_b$ is the \emph{hooklength} of $b$, i.e., the number of boxes weakly right and strictly below $b$.
\item $a_{w,\lambda}$ counts \emph{EG tableaux}: row and column increasing fillings $T$ of $\lambda$ such that reading the entries 
$(i_1,i_2,\ldots, i_{|\lambda|})$ of $T$ along columns, top to bottom, and right to left, gives a reduced word $s_{i_1}\cdots s_{i_{|\lambda|}}$ for $w$ (cf.~\cite{BKSTY}).
\end{itemize}

Let $w_0=n \ n-1 \ n-2 \ \ldots \ 3 \ 2 \ 1$ be the unique longest length
permutation of $S_n$ (hence $\ell(w_0)={n\choose 2}$). R.~P.~Stanley \cite{Stanley1984} proved that,
in this case, (\ref{eqn:stanleyformula}) is short:
\begin{equation}
\label{eqn:w0thing}
\#{\sf Red}(w_0)=f^{(n-1,n-2,\ldots,3,2,1)};
\end{equation}
hence $\#{\sf Red}(w_0)$ is computed by (\ref{eqn:HLF}). 

One measure of the brevity of (\ref{eqn:stanleyformula}) is the \emph{Edelman-Greene statistic} on $S_n$, 
\[{\sf EG}(w)=\sum_{\lambda} a_{w,\lambda};\]
this was introduced by S.~Billey-B.~Pawlowski \cite{Billey}.
From (\ref{eqn:w0thing}), one sees ${\sf EG}(w_0)=1$.
Permutations $w$ such that ${\sf EG}(w)=1$ are \emph{vexillary}. These permutations are characterized by \emph{$2143$-pattern avoidance}: there are no indices $i_1<i_2<i_3<i_4$
such that 
$w(i_1),w(i_2),w(i_3),w(i_4)$ are in the same relative
order as $2143$. For instance,
$w=\underline{5}\underline{4}27\underline{8}31\underline{6}$ is not vexillary; the underlined positions give a $2143$ pattern.
Each such $w$ has \emph{shape} $\lambda(w)$ (defined in Section~\ref{sec:2.2}). Extending (\ref{eqn:w0thing}), whenever $w$ is vexillary,
\begin{equation}
\label{eqn:whenvex}
\#{\sf Red}(w)=f^{\lambda(w)};
\end{equation}
see, e.g., \cite[Corollary~2.8.2]{Manivel}. Our main result  (Theorem~\ref{thm:firstmain}) is that {\sf EG} is typically large. This implies
a (weak version) of a Theorem of B.~Pawlowski \cite[Theorem~3.2.7]{Brendan:thesis}:

\begin{theorem}[Average exponential growth]
\label{thm:averagemain}
${\mathbb E}[{\sf EG}]=\Omega(c^n)$, 
for some fixed constant $c>1$. 
\end{theorem}

\subsection{Computational complexity and transition}
Our proof of Theorem~\ref{thm:averagemain} applies the \emph{transition algorithm} of A.~Lascoux-M.~P.~Sch\"utzenberger \cite{LS:transition}
(cf.~\cite[Sections~2.7, 2.8]{Manivel}). This algorithm constructs 
 a tree ${\mathcal T}(w)$ whose  root is $w$ and the leaves ${\mathcal L}(w)$ are labelled with
vexillary permutations (with multiplicity). With this, 
\begin{equation}
\label{eqn:transitionexp}
\#{\sf Red}(w)=\sum_{v\in {\mathcal L}(w)} f^{\lambda(v)};
\end{equation}
see Section~\ref{sec:2} for details. Different $v$ may give the same $\lambda(v)$. After combining such terms, (\ref{eqn:transitionexp}) is the same as (\ref{eqn:stanleyformula}); see Lemma~\ref{lemma:basis}.

The (practical) efficiency of (extensions/variations of) transition has been mentioned a number of times.  
S.~Billey \cite{Billey:transition} calls transition ``one of the most efficient methods'' to compute Schubert polynomials. 
See also A.~Buch \cite[Section~3.4]{Buch:qhpartial} and Z.~Hamaker-E.~Marburg-B.~Pawlowski~\cite{Hamaker:SchurP}. On the other hand, concerning the application of transition to computing the Littlewood-Richardson coefficients \cite{LS:transition}, A.~Garsia \cite[p.~52]{Garsia} writes:
\begin{quotation}
``Curiously,
their algorithm (in spite of their claims to
the contrary) is hopelessly inefficient as
compared with well known methods.''
\end{quotation}
He also refers to transition as ``efficient'' for a different purpose in his study of ${\sf Red}(w)$.

Theorem~\ref{thm:averagemain} is actually a reformulation of the following result 
which is a \emph{formal} complexity analysis of transition:
\begin{theorem}
\label{thm:exponentialintro}
${\mathbb E}(\#\mathcal L)=\Omega(c^n)$ for a fixed constant $c>1$. That is the average running time 
of transition, as an algorithm to compute $\#{\sf Red}(w)$, is at least exponential in $n$.
\end{theorem}

Theorem~\ref{thm:firstmain} strengthens Theorem~\ref{thm:exponentialintro} to show that the ``typical'' running time is exponentially large.
To prove Theorem~\ref{thm:exponentialintro}
  we use that the expected number of occurences of a fixed pattern $\pi\in {S}_k$ in $w\in{S}_n$ is ${n\choose k}/k!$. Thus for $u=2143$, this expectation is $O(n^4)$. 
  One shows each step of transition reduces the number of $2143$ patterns by $O(n^3)$. Using the graphical description of 
transition by A.~Knutson and the third author \cite{Knutson.Yong}, a node $u$ of ${\mathcal T}(w)$ has exactly one child $u'$ only if $u'$ has weakly more $2143$ patterns than $u$ does. Consequently, ${\mathcal T}(w)$ has $\Omega(n)$ branch points along any root-to-leaf path and 
thus exponentially many leaves.  (In fact, the $c>1$ from our argument is close 
to $1$.)

Of course, the exponential average run-time of transition does not imply computing $\#{\sf Red}(w)$ is 
hard. Suppose one encodes a permutation $w$ by its \emph{Lehmer code} ${\sf code}(w)=(c_1,c_2,\ldots,c_L)$. \emph{What is the worst case complexity of computing $\#{\sf Red}(w)$ given input ${\sf code}(w)$?}

L.~Valiant \cite{Valiant} introduced the complexity class $\#{\sf P}$ of  
problems that count the number of accepting paths of a non-deterministic Turing machine running in polynomial time in the 
length of the input. Let ${\sf FP}$ be the class of function problems solvable in polynomial time on a deterministic Turing machine. It is basic theory that ${\sf FP}\subseteq \#{\sf P}$.
 \begin{observation}
 \label{obs:complex}
$\#{\sf Red}(w)\not\in \#{\sf P}$. In particular, 
$\#{\sf Red}(w)\not\in {\sf FP}$. 
 \end{observation}
 \begin{proof}
Let $\theta_n$ be the vexillary permutation with ${\sf code}(\theta_n)=(n,n)$. Then, using (\ref{eqn:whenvex}), 
\begin{equation}
\label{eqn:Catalan}
\#{\sf Red}(\theta_n)=f^{(n,n)}=C_n:=\frac{1}{n+1}{2n\choose n}.
\end{equation}
The middle equality is textbook: there is a bijection between standard Young tableaux
of shape $(n,n)$ and Dyck paths from $(0,0)$ to $(2n,0)$; both are enumerated by the \emph{Catalan
number} $C_n$. Now, $\#{\sf Red}(\theta_n)$ is \emph{doubly} exponential in the input length $O(\log n)$. No such problem can be in $\#{\sf P}$ \cite[Section~3]{Narayanan}. 
($\#{\sf Red}(w)\not\in {\sf FP}$ is true from this argument for the simple reason that  
it takes exponential time just to write down the output.)
\end{proof}
 
A counting problem ${\mathcal P}$ is $\#{\sf P}$-{\sf hard} if any problem in $\#{\sf P}$ has a polynomial-time counting reduction to ${\mathcal P}$. \emph{Is $\#{\sf Red}(w)\in \#{\sf P}$-{\sf hard}?} 

Observation~\ref{obs:complex} is
dependent on the choice of encoding. For example, if one encodes a permutation $w\in S_n$ in the inefficient one-line notation, the input takes $O(n\log n)$ space. Since
$\ell(w)\leq {n\choose 2}$ is polynomial in the input length, it follows that $\#{\sf Red}(w)\in \#{\sf P}$; see \cite{Samuels}. 

\begin{problem}
\label{conj:Dec16conj}
Does there exist an $n^{O(1)}$-algorithm to compute $\#{\sf Red}(w)$?
\end{problem}
It is easy to see that $\#{\sf Red}(u)\leq \#{\sf Red}(us_i)$ whenever $\ell(us_i)=\ell(u)+1$. Hence,
$\#{\sf Red}(w)$ is maximized at $w=w_0$. So, by (\ref{eqn:w0thing}), 
$\log (\#{\sf Red}(w))\in n^{O(1)}$. Thus,
unlike Observation~\ref{obs:complex}, there is no easy negative solution to Problem~\ref{conj:Dec16conj} (and any negative solution implies ${\sf FP}\neq \#{\sf P}$, which is a famous open problem). Indeed, in the vexillary case (\ref{eqn:whenvex}), the hook-length formula (\ref{eqn:HLF}) gives a $n^{O(1)}$-algorithm for $\#{\sf Red}(w)$. 

\subsection{Hecke words}
Section~\ref{sec:4} studies the more general problem of counting ${\sf Hecke}(w,N)$, the set of \emph{Hecke words} of length $N$ whose \emph{Demazure product} is a given $w\in {S}_n$. Here, the role of Stanley's symmetric polynomial is played by the \emph{stable Grothendieck polynomial} defined by S.~Fomin and A.~N.~Kirillov \cite{Fomin.Kirillov}. Using work of S.~Fomin and C.~Greene \cite{Fomin.Greene} and of A.~Buch, A.~Kresch, M.~Shimozono, H.~Tamvakis and the third author \cite{BKSTY},
one has two analogues of the results of Edelman-Greene \cite{Edelman.Greene}. However,  useful enumeration \emph{formulas} for Hecke words, even when $w$ is vexillary, is a challenge.

As explained by Proposition~\ref{prop:secondG}, enumerating Hecke words is closely related to the
problem of counting $f^{\lambda,N}$, the number of \emph{set-valued tableaux} \cite{Buch:KLR} that are
\emph{$N$-standard} of shape $\lambda$. These
are fillings $T$ of
the boxes of $\lambda$ by $1,2,\ldots,N$, where each entry appears exactly once, and if one chooses precisely one entry from each box of $T$, one obtains a semistandard tableau.  For example,
if $N=8$ and $\lambda=(3,2)$, one tableau is 
\ $\tableau{\scalebox{0.7}{1,2} & \scalebox{0.7}{4,5} &  8\\
3 & \scalebox{0.7}{6,7}}$.

By Observation~\ref{obs:complex}'s reasoning, (\ref{eqn:Catalan}) shows there is no algorithm 
to compute $f^{\lambda,N}$ that is
polynomial-time in the bit-length of the input $(\lambda,N)$. 

\begin{problem}
\label{prob:Dec21abc}
Does there exist an algorithm to compute $f^{\lambda,N}$ that is polynomial in $|\lambda|$ and $N$?
\end{problem}

Clearly, (\ref{eqn:HLF}) gives a solution
when $N=|\lambda|$. 
Using a theorem of C.~Lenart \cite{Lenart}, there exists an $|\lambda|^{O(1)}$ algorithm 
for any $\lambda$ and where $N= |\lambda|+k$, if $k$ is \emph{fixed} (Proposition~\ref{prop:plus2}).

Recent 
work of M.~Chan-N.~Pflueger \cite{Chan} and D.~Anderson-L.~Chen-N.~Tarasca \cite{Anderson} motivates 
study of $f^{\lambda,N}$ in terms of
\emph{Brill-Noether varieties}. We remark on two manifestly nonnegative formulas for the Euler characteristics of these varieties (Corollary~\ref{cor:Euler}).

\subsection{Randomization} Section~\ref{sec:5} gives three randomized algorithms to estimate 
$\#{\sf Red}(w)$ and/or $\#{\sf Hecke}(w,N)$
using \emph{importance sampling}. That is, let $S$ be a finite set. Assign $s\in S$ probability $p_s$. Let ${\sf Z}$ be a random variable on $S$ with $Z(s)=1/p_s$. Then 
${\mathbb E}({\sf Z})=\sum_{s\in S} p_s\times \frac{1}{p_s}=\#S$.
Using this, one can devise simple Monte Carlo algorithms to estimate $\#S$. The idea goes back to at least a 1951 article of H.~Kahn-T.~E.~Harris
\cite{Kahn.Harris}, who furthermore credit J.~von Neumann. The application to combinatorial enumeration
was popularized through D.~Knuth's article \cite{Knuth:science} which applies it to 
estimating the number of self-avoiding walks in a grid. An application to approximating the \emph{permanent} was given by L.~E.~Rasmussen \cite{Rasmussen}. More recently, J.~Blitzstein-P.~Diaconis \cite{Diaconisblitz} develop an importance sampling algorithm to estimate
the number of graphs with a given degree sequence. We are suggesting another avenue of applicability, to core objects of algebraic combinatorics.

\section{The Graphical Transition algorithm}
\label{sec:2}
\subsection{Preliminaries}
The \emph{graph} $G(w)$ of a
permutation $w\in {S}_n$ is the $n\times n$ grid, with a $\bullet$ placed
in position $(i,w(i))$ (in matrix coordinates).
The \emph{Rothe diagram}
of $w$ is given by 
\[D(w)=\{(i,j): 1\leq i,j\leq n, \ j<w(i), \ i<w^{-1}(j)\}.\]
Pictorially, this is described by
striking out boxes below and to the right of each $\bullet$ in $G(w)$. $D(w)$ consists of the remaining boxes. 
If it exists, the connected component involving $(1,1)$ is the \emph{dominant component}. 
The \emph{essential set} of $w$ 
consists of the maximally southeast boxes of each
connected component of $D(w)$, i.e.,
\[{\mathcal Ess}(w)=\{(i,j)\in D(w): (i+1,j),(i,j+1)\not\in D(w)\}.\]
If it exists, the \emph{accessible box} is the southmost then eastmost essential set box \emph{not in the dominant component}. 
For example, if 
$w=54278316 \in {S}_{8}$, $D(w)$ is depicted by: 
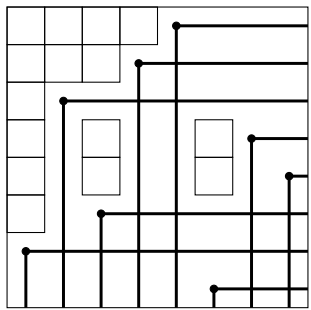
\[
\begin{tikzpicture}[scale=.5]
\draw (0,0) rectangle (8,8);

\draw (0,7) rectangle (1,8);
\draw (1,7) rectangle (2,8);
\draw (2,7) rectangle (3,8);
\draw (3,7) rectangle (4,8);
\draw (0,6) rectangle (1,7);
\draw (1,6) rectangle (2,7);
\draw (2,6) rectangle (3,7);
\draw (0,5) rectangle (1,6);
\draw (0,4) rectangle (1,5);
\draw (2,4) rectangle (3,5);
\draw (5,4) rectangle (6,5);
\draw (0,3) rectangle (1,4);
\draw (2,3) rectangle (3,4);
\draw (5,3) rectangle (6,4);
\draw (0,2) rectangle (1,3);

\filldraw (4.5,7.5) circle (.5ex);
\draw[line width = .2ex] (4.5,0) -- (4.5,7.5) -- (8,7.5);
\filldraw (3.5,6.5) circle (.5ex);
\draw[line width = .2ex] (3.5,0) -- (3.5,6.5) -- (8,6.5);
\filldraw (1.5,5.5) circle (.5ex);
\draw[line width = .2ex] (1.5,0) -- (1.5,5.5) -- (8,5.5);
\filldraw (6.5,4.5) circle (.5ex);
\draw[line width = .2ex] (6.5,0) -- (6.5,4.5) -- (8,4.5);
\filldraw (7.5,3.5) circle (.5ex);
\draw[line width = .2ex] (7.5,0) -- (7.5,3.5) -- (8,3.5);
\filldraw (2.5,2.5) circle (.5ex);
\draw[line width = .2ex] (2.5,0) -- (2.5,2.5) -- (8,2.5);
\filldraw (0.5,1.5) circle (.5ex);
\draw[line width = .2ex] (0.5,0) -- (0.5,1.5) -- (8,1.5);
\filldraw (5.5,0.5) circle (.5ex);
\draw[line width = .2ex] (5.5,0) -- (5.5,0.5) -- (8,0.5);
\end{tikzpicture}
\]
Also,
${\mathcal Ess}(w)=\{(1,4), (2,3),(5,3), (5,6), (6,1)\}$,
and the accessible box is at $(5,6)$. 
The \emph{Lehmer code} of $w\in S_{\infty}$, denoted ${\sf code}(w)$ is the vector $(c_1,c_2,\ldots,c_L)$ where $c_i$ equals the number of boxes in row $i$ of the Rothe
diagram of $w$. We will assume $L$ is minimum (i.e., ${\sf code}(w)$ does not have trailing zeros). By this convention, ${\sf code}(id)=()$.

 \emph{Fulton's criterion} \cite[Remark~9.17]{Fulton} states that 
$u$ is vexillary if and only if there does not exist two essential set boxes where one is strictly northwest of the other. Thus, using the above picture of $D(w)$ we can see that $w$ is not vexillary because of, e.g., $(1,4)$ and $(5,6)$.

\subsection{Description of ${\mathcal T}(w)$} \label{sec:2.2}
The original description of transition was given in \cite{LS:transition}; this account is also given an exposition in \cite[Sections~2.7.3, 2.7.4, 2.8.1]{Manivel}. We follow the graphical description given in \cite{Knutson.Yong}
and its elaboration in \cite{AdveRobichauxYong}.
There are some minor choices in describing the transition tree, and those
of \cite{Knutson.Yong, AdveRobichauxYong} differ slightly from \cite{LS:transition,Manivel}. 

We describe the graphical version of the transition algorithm to compute $\#{\sf Red}(w)$. The root of the tree
is labelled by $D(w)$. If $w$ is vexillary, stop. Otherwise, there exists an accessible box. (If not, $D(w)$ consists only of the dominant component and, by Fulton's criterion, $w$ is vexillary, a contradiction.) The \emph{pivots} of $D(w)$ are the maximally southeast
$\bullet$'s of $G(w)$, say $b_1,b_2,\ldots,b_t$ that are 
northwest of the accessible box $e$. 

If $w$ is not vexillary, the children of $w$ are defined
as follows. For each $i=1,2,\ldots,t$, let $R_i$ be the rectangle defined by 
$b_i$ and $e$. Remove $b_i$ and its rays from $G(w)$ to form $G^{(i)}(w)$. Order the boxes $\{v_i\}_{i=1}^{m}$ in English reading order. Move
$v_1$ strictly north and strictly west to the closest position not occupied by another box of $D(w)$ or a ray from $G^{(i)}(w)$. Now, iterate this 
procedure with $v_2,v_3,\ldots$. At each step, $v_j$ may move to a position vacated by earlier moves. The result is the diagram $D(w^{(i)})$ of some permutation $w^{(i)}$.
These $D(w^{(i)})$'s are the children of $D(w)$. We call the transformation $D(w)\to D(w^{(i)})$ a \emph{marching move}.

\begin{example}  Continuing our example, the pivots of $w$ are $(1,5), (2,4)$ and $(3,2)$.  We now obtain the child
corresponding to the pivot $b_2=(2,4)$:
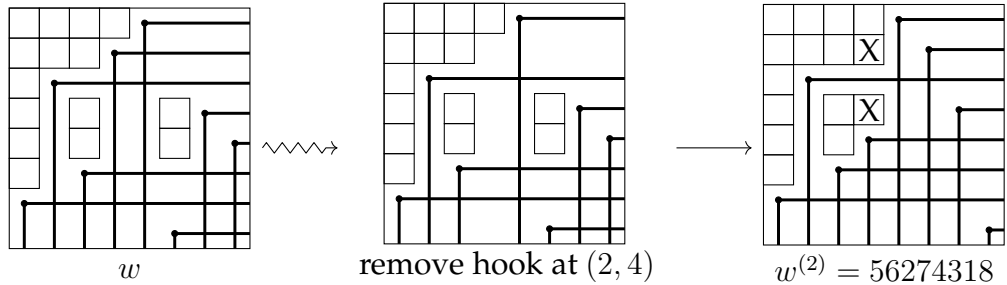
\begin{center}
\begin{tikzpicture}
\node (root) {\begin{tikzpicture}[scale=.4]
\draw (0,0) rectangle (8,8);

\draw (0,7) rectangle (1,8);
\draw (1,7) rectangle (2,8);
\draw (2,7) rectangle (3,8);
\draw (3,7) rectangle (4,8);
\draw (0,6) rectangle (1,7);
\draw (1,6) rectangle (2,7);
\draw (2,6) rectangle (3,7);
\draw (0,5) rectangle (1,6);
\draw (0,4) rectangle (1,5);
\draw (2,4) rectangle (3,5);
\draw (5,4) rectangle (6,5);
\draw (0,3) rectangle (1,4);
\draw (2,3) rectangle (3,4);
\draw (5,3) rectangle (6,4);
\draw (0,2) rectangle (1,3);

\node at (4,-.75) {$w$};

\filldraw (4.5,7.5) circle (.5ex);
\draw[line width = .2ex] (4.5,0) -- (4.5,7.5) -- (8,7.5);
\filldraw (3.5,6.5) circle (.5ex);
\draw[line width = .2ex] (3.5,0) -- (3.5,6.5) -- (8,6.5);
\filldraw (1.5,5.5) circle (.5ex);
\draw[line width = .2ex] (1.5,0) -- (1.5,5.5) -- (8,5.5);
\filldraw (6.5,4.5) circle (.5ex);
\draw[line width = .2ex] (6.5,0) -- (6.5,4.5) -- (8,4.5);
\filldraw (7.5,3.5) circle (.5ex);
\draw[line width = .2ex] (7.5,0) -- (7.5,3.5) -- (8,3.5);
\filldraw (2.5,2.5) circle (.5ex);
\draw[line width = .2ex] (2.5,0) -- (2.5,2.5) -- (8,2.5);
\filldraw (0.5,1.5) circle (.5ex);
\draw[line width = .2ex] (0.5,0) -- (0.5,1.5) -- (8,1.5);
\filldraw (5.5,0.5) circle (.5ex);
\draw[line width = .2ex] (5.5,0) -- (5.5,0.5) -- (8,0.5);
\end{tikzpicture}};

\node[right=1 of root] (temp) {\begin{tikzpicture}[scale=.4]
\draw (0,0) rectangle (8,8);

\draw (0,7) rectangle (1,8);
\draw (1,7) rectangle (2,8);
\draw (2,7) rectangle (3,8);
\draw (3,7) rectangle (4,8);
\draw (0,6) rectangle (1,7);
\draw (1,6) rectangle (2,7);
\draw (2,6) rectangle (3,7);
\draw (0,5) rectangle (1,6);
\draw (0,4) rectangle (1,5);
\draw (2,4) rectangle (3,5);
\draw (5,4) rectangle (6,5);
\draw (0,3) rectangle (1,4);
\draw (2,3) rectangle (3,4);
\draw (5,3) rectangle (6,4);
\draw (0,2) rectangle (1,3);

\node at (-1.2,-.75) {remove hook at $(2,4)$};

\filldraw (4.5,7.5) circle (.5ex);
\draw[line width = .2ex] (4.5,0) -- (4.5,7.5) -- (8,7.5);
\filldraw (1.5,5.5) circle (.5ex);
\draw[line width = .2ex] (1.5,0) -- (1.5,5.5) -- (8,5.5);
\filldraw (6.5,4.5) circle (.5ex);
\draw[line width = .2ex] (6.5,0) -- (6.5,4.5) -- (8,4.5);
\filldraw (7.5,3.5) circle (.5ex);
\draw[line width = .2ex] (7.5,0) -- (7.5,3.5) -- (8,3.5);
\filldraw (2.5,2.5) circle (.5ex);
\draw[line width = .2ex] (2.5,0) -- (2.5,2.5) -- (8,2.5);
\filldraw (0.5,1.5) circle (.5ex);
\draw[line width = .2ex] (0.5,0) -- (0.5,1.5) -- (8,1.5);
\filldraw (5.5,0.5) circle (.5ex);
\draw[line width = .2ex] (5.5,0) -- (5.5,0.5) -- (8,0.5);
\end{tikzpicture}};

\node[right=1 of temp] (ans) {
\begin{tikzpicture}[scale=.4]
\draw (0,0) rectangle (8,8);

\draw (0,7) rectangle (1,8);
\draw (1,7) rectangle (2,8);
\draw (2,7) rectangle (3,8);
\draw (3,7) rectangle (4,8);
\draw (0,6) rectangle (1,7);
\draw (1,6) rectangle (2,7);
\draw (2,6) rectangle (3,7);
\draw (3,6) rectangle (4,7);
\draw (0,5) rectangle (1,6);
\draw (0,4) rectangle (1,5);
\draw (2,4) rectangle (3,5);
\draw (3,4) rectangle (4,5);
\draw (0,3) rectangle (1,4);
\draw (2,3) rectangle (3,4);
\draw (0,2) rectangle (1,3);

\node at (2.8,6.5) {X};
\node at (2.8,4.5) {X};
\node at (0,-.75) {$w^{(2)}=56274318$};

\filldraw (4.5,7.5) circle (.5ex);
\draw[line width = .2ex] (4.5,0) -- (4.5,7.5) -- (8,7.5);
\filldraw (5.5,6.5) circle (.5ex);
\draw[line width = .2ex] (5.5,0) -- (5.5,6.5) -- (8,6.5);
\filldraw (1.5,5.5) circle (.5ex);
\draw[line width = .2ex] (1.5,0) -- (1.5,5.5) -- (8,5.5);
\filldraw (6.5,4.5) circle (.5ex);
\draw[line width = .2ex] (6.5,0) -- (6.5,4.5) -- (8,4.5);
\filldraw (3.5,3.5) circle (.5ex);
\draw[line width = .2ex] (3.5,0) -- (3.5,3.5) -- (8,3.5);
\filldraw (2.5,2.5) circle (.5ex);
\draw[line width = .2ex] (2.5,0) -- (2.5,2.5) -- (8,2.5);
\filldraw (0.5,1.5) circle (.5ex);
\draw[line width = .2ex] (0.5,0) -- (0.5,1.5) -- (8,1.5);
\filldraw (7.5,0.5) circle (.5ex);
\draw[line width = .2ex] (7.5,0) -- (7.5,0.5) -- (8,0.5);
\end{tikzpicture}};

\draw [->,
line join=round,
decorate, decoration={
    zigzag,
    segment length=6,
    amplitude=2,post=lineto,
    post length=2pt
}] (root) -- (temp);
\path (temp) edge[->] node[midway,above] {\ } (ans);

\end{tikzpicture}
\end{center}
We have indicated by ``X'' the boxes that have moved. This process constructs one of the three children of $w$. In Figure~\ref{fig:transTree} we draw the remainder of ${\mathcal T}(w)$.\qed
\end{example}

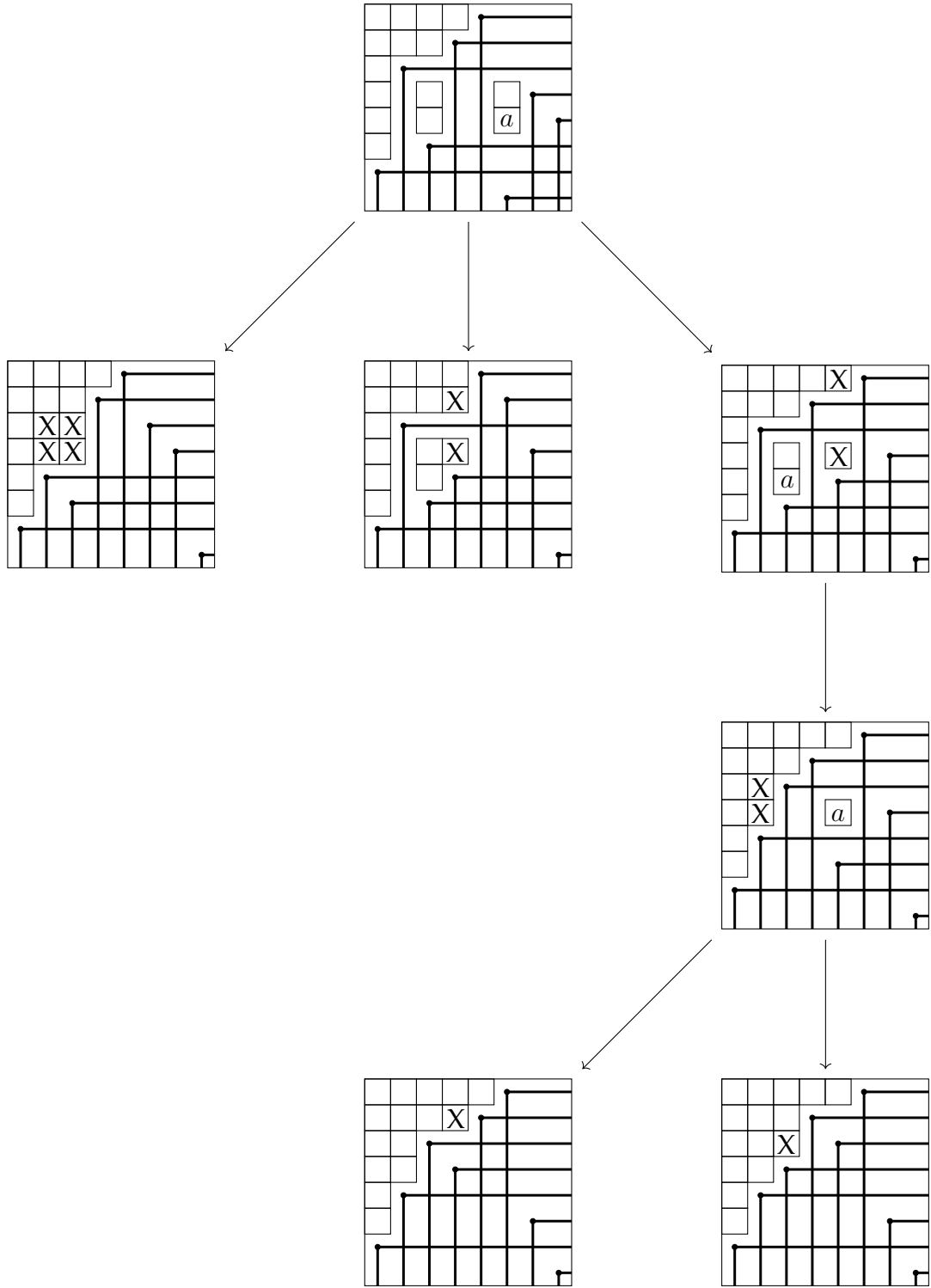
\begin{figure}
\begin{tikzpicture}

\node (root) {\begin{tikzpicture}[scale=.4]
\draw (0,0) rectangle (8,8);

\draw (0,7) rectangle (1,8);
\draw (1,7) rectangle (2,8);
\draw (2,7) rectangle (3,8);
\draw (3,7) rectangle (4,8);
\draw (0,6) rectangle (1,7);
\draw (1,6) rectangle (2,7);
\draw (2,6) rectangle (3,7);
\draw (0,5) rectangle (1,6);
\draw (0,4) rectangle (1,5);
\draw (2,4) rectangle (3,5);
\draw (5,4) rectangle (6,5);
\draw (0,3) rectangle (1,4);
\draw (2,3) rectangle (3,4);
\draw (5,3) rectangle (6,4) node[pos=.5] {$a$};
\draw (0,2) rectangle (1,3);

\filldraw (4.5,7.5) circle (.5ex);
\draw[line width = .2ex] (4.5,0) -- (4.5,7.5) -- (8,7.5);
\filldraw (3.5,6.5) circle (.5ex);
\draw[line width = .2ex] (3.5,0) -- (3.5,6.5) -- (8,6.5);
\filldraw (1.5,5.5) circle (.5ex);
\draw[line width = .2ex] (1.5,0) -- (1.5,5.5) -- (8,5.5);
\filldraw (6.5,4.5) circle (.5ex);
\draw[line width = .2ex] (6.5,0) -- (6.5,4.5) -- (8,4.5);
\filldraw (7.5,3.5) circle (.5ex);
\draw[line width = .2ex] (7.5,0) -- (7.5,3.5) -- (8,3.5);
\filldraw (2.5,2.5) circle (.5ex);
\draw[line width = .2ex] (2.5,0) -- (2.5,2.5) -- (8,2.5);
\filldraw (0.5,1.5) circle (.5ex);
\draw[line width = .2ex] (0.5,0) -- (0.5,1.5) -- (8,1.5);
\filldraw (5.5,0.5) circle (.5ex);
\draw[line width = .2ex] (5.5,0) -- (5.5,0.5) -- (8,0.5);
\end{tikzpicture}};
\node[below left = 2 and 2 of root] (c1) {\begin{tikzpicture}[scale=.4]
\draw (0,0) rectangle (8,8);

\draw (0,7) rectangle (1,8);
\draw (1,7) rectangle (2,8);
\draw (2,7) rectangle (3,8);
\draw (3,7) rectangle (4,8);
\draw (0,6) rectangle (1,7);
\draw (1,6) rectangle (2,7);
\draw (2,6) rectangle (3,7);
\draw (0,5) rectangle (1,6);
\draw (1,5) rectangle (2,6) node[pos=1.25] {X};
\draw (2,5) rectangle (3,6) node[pos=1.25] {X};
\draw (0,4) rectangle (1,5);
\draw (1,4) rectangle (2,5) node[pos=1.25] {X};
\draw (2,4) rectangle (3,5) node[pos=1.25] {X};
\draw (0,3) rectangle (1,4);
\draw (0,2) rectangle (1,3);

\filldraw (4.5,7.5) circle (.5ex);
\draw[line width = .2ex] (4.5,0) -- (4.5,7.5) -- (8,7.5);
\filldraw (3.5,6.5) circle (.5ex);
\draw[line width = .2ex] (3.5,0) -- (3.5,6.5) -- (8,6.5);
\filldraw (5.5,5.5) circle (.5ex);
\draw[line width = .2ex] (5.5,0) -- (5.5,5.5) -- (8,5.5);
\filldraw (6.5,4.5) circle (.5ex);
\draw[line width = .2ex] (6.5,0) -- (6.5,4.5) -- (8,4.5);
\filldraw (1.5,3.5) circle (.5ex);
\draw[line width = .2ex] (1.5,0) -- (1.5,3.5) -- (8,3.5);
\filldraw (2.5,2.5) circle (.5ex);
\draw[line width = .2ex] (2.5,0) -- (2.5,2.5) -- (8,2.5);
\filldraw (0.5,1.5) circle (.5ex);
\draw[line width = .2ex] (0.5,0) -- (0.5,1.5) -- (8,1.5);
\filldraw (7.5,0.5) circle (.5ex);
\draw[line width = .2ex] (7.5,0) -- (7.5,0.5) -- (8,0.5);
\end{tikzpicture}
};
\node[below = 2 of root] (c2) {\begin{tikzpicture}[scale=.4]
\draw (0,0) rectangle (8,8);

\draw (0,7) rectangle (1,8);
\draw (1,7) rectangle (2,8);
\draw (2,7) rectangle (3,8);
\draw (3,7) rectangle (4,8);
\draw (0,6) rectangle (1,7);
\draw (1,6) rectangle (2,7);
\draw (2,6) rectangle (3,7);
\draw (3,6) rectangle (4,7);
\draw (0,5) rectangle (1,6);
\draw (0,4) rectangle (1,5);
\draw (2,4) rectangle (3,5);
\draw (3,4) rectangle (4,5);
\draw (0,3) rectangle (1,4);
\draw (2,3) rectangle (3,4);
\draw (0,2) rectangle (1,3);

\node at (3.5,7.2) {X};
\node at (3.5,5.2) {X};

\filldraw (4.5,7.5) circle (.5ex);
\draw[line width = .2ex] (4.5,0) -- (4.5,7.5) -- (8,7.5);
\filldraw (5.5,6.5) circle (.5ex);
\draw[line width = .2ex] (5.5,0) -- (5.5,6.5) -- (8,6.5);
\filldraw (1.5,5.5) circle (.5ex);
\draw[line width = .2ex] (1.5,0) -- (1.5,5.5) -- (8,5.5);
\filldraw (6.5,4.5) circle (.5ex);
\draw[line width = .2ex] (6.5,0) -- (6.5,4.5) -- (8,4.5);
\filldraw (3.5,3.5) circle (.5ex);
\draw[line width = .2ex] (3.5,0) -- (3.5,3.5) -- (8,3.5);
\filldraw (2.5,2.5) circle (.5ex);
\draw[line width = .2ex] (2.5,0) -- (2.5,2.5) -- (8,2.5);
\filldraw (0.5,1.5) circle (.5ex);
\draw[line width = .2ex] (0.5,0) -- (0.5,1.5) -- (8,1.5);
\filldraw (7.5,0.5) circle (.5ex);
\draw[line width = .2ex] (7.5,0) -- (7.5,0.5) -- (8,0.5);
\end{tikzpicture}};
\node[below right = 2 and 2 of root] (c3) {\begin{tikzpicture}[scale=.4]
\draw (0,0) rectangle (8,8);

\draw (0,7) rectangle (1,8);
\draw (1,7) rectangle (2,8);
\draw (2,7) rectangle (3,8);
\draw (3,7) rectangle (4,8);
\draw (4,7) rectangle (5,8);
\draw (0,6) rectangle (1,7);
\draw (1,6) rectangle (2,7);
\draw (2,6) rectangle (3,7);
\draw (0,5) rectangle (1,6);
\draw (0,4) rectangle (1,5);
\draw (2,4) rectangle (3,5);
\draw (4,4) rectangle (5,5);
\draw (0,3) rectangle (1,4);
\draw (2,3) rectangle (3,4);
\draw (0,2) rectangle (1,3);

\node at (3.8,8.2) {X};
\node at (3.8,5.2) {X};
\node at (1.9,4.1) {$a$};

\filldraw (5.5,7.5) circle (.5ex);
\draw[line width = .2ex] (5.5,0) -- (5.5,7.5) -- (8,7.5);
\filldraw (3.5,6.5) circle (.5ex);
\draw[line width = .2ex] (3.5,0) -- (3.5,6.5) -- (8,6.5);
\filldraw (1.5,5.5) circle (.5ex);
\draw[line width = .2ex] (1.5,0) -- (1.5,5.5) -- (8,5.5);
\filldraw (6.5,4.5) circle (.5ex);
\draw[line width = .2ex] (6.5,0) -- (6.5,4.5) -- (8,4.5);
\filldraw (4.5,3.5) circle (.5ex);
\draw[line width = .2ex] (4.5,0) -- (4.5,3.5) -- (8,3.5);
\filldraw (2.5,2.5) circle (.5ex);
\draw[line width = .2ex] (2.5,0) -- (2.5,2.5) -- (8,2.5);
\filldraw (0.5,1.5) circle (.5ex);
\draw[line width = .2ex] (0.5,0) -- (0.5,1.5) -- (8,1.5);
\filldraw (7.5,0.5) circle (.5ex);
\draw[line width = .2ex] (7.5,0) -- (7.5,0.5) -- (8,0.5);
\end{tikzpicture}};
\node[below= 2  of c3] (c4) {\begin{tikzpicture}[scale=.4]
\draw (0,0) rectangle (8,8);

\draw (0,7) rectangle (1,8);
\draw (1,7) rectangle (2,8);
\draw (2,7) rectangle (3,8);
\draw (3,7) rectangle (4,8);
\draw (4,7) rectangle (5,8);
\draw (0,6) rectangle (1,7);
\draw (1,6) rectangle (2,7);
\draw (2,6) rectangle (3,7);
\draw (0,5) rectangle (1,6);
\draw (1,5) rectangle (2,6);
\draw (0,4) rectangle (1,5);
\draw (1,4) rectangle (2,5);
\draw (4,4) rectangle (5,5);
\draw (0,3) rectangle (1,4);
\draw (0,2) rectangle (1,3);

\node at (1.5,6.2) {X};
\node at (1.5,5.2) {X};
\node at (4.5,5) {$a$};

\filldraw (5.5,7.5) circle (.5ex);
\draw[line width = .2ex] (5.5,0) -- (5.5,7.5) -- (8,7.5);
\filldraw (3.5,6.5) circle (.5ex);
\draw[line width = .2ex] (3.5,0) -- (3.5,6.5) -- (8,6.5);
\filldraw (2.5,5.5) circle (.5ex);
\draw[line width = .2ex] (2.5,0) -- (2.5,5.5) -- (8,5.5);
\filldraw (6.5,4.5) circle (.5ex);
\draw[line width = .2ex] (6.5,0) -- (6.5,4.5) -- (8,4.5);
\filldraw (1.5,3.5) circle (.5ex);
\draw[line width = .2ex] (1.5,0) -- (1.5,3.5) -- (8,3.5);
\filldraw (4.5,2.5) circle (.5ex);
\draw[line width = .2ex] (4.5,0) -- (4.5,2.5) -- (8,2.5);
\filldraw (0.5,1.5) circle (.5ex);
\draw[line width = .2ex] (0.5,0) -- (0.5,1.5) -- (8,1.5);
\filldraw (7.5,0.5) circle (.5ex);
\draw[line width = .2ex] (7.5,0) -- (7.5,0.5) -- (8,0.5);
\end{tikzpicture}};
\node[below = 2  of c4] (c5) {\begin{tikzpicture}[scale=.4]
\draw (0,0) rectangle (8,8);

\draw (0,7) rectangle (1,8);
\draw (1,7) rectangle (2,8);
\draw (2,7) rectangle (3,8);
\draw (3,7) rectangle (4,8);
\draw (4,7) rectangle (5,8);
\draw (0,6) rectangle (1,7);
\draw (1,6) rectangle (2,7);
\draw (2,6) rectangle (3,7);
\draw (0,5) rectangle (1,6);
\draw (1,5) rectangle (2,6);
\draw (2,5) rectangle (3,6);
\draw (0,4) rectangle (1,5);
\draw (1,4) rectangle (2,5);
\draw (0,3) rectangle (1,4);
\draw (0,2) rectangle (1,3);

\node at (2.5,6.2) {X};

\filldraw (5.5,7.5) circle (.5ex);
\draw[line width = .2ex] (5.5,0) -- (5.5,7.5) -- (8,7.5);
\filldraw (3.5,6.5) circle (.5ex);
\draw[line width = .2ex] (3.5,0) -- (3.5,6.5) -- (8,6.5);
\filldraw (4.5,5.5) circle (.5ex);
\draw[line width = .2ex] (4.5,0) -- (4.5,5.5) -- (8,5.5);
\filldraw (2.5,4.5) circle (.5ex);
\draw[line width = .2ex] (2.5,0) -- (2.5,4.5) -- (8,4.5);
\filldraw (1.5,3.5) circle (.5ex);
\draw[line width = .2ex] (1.5,0) -- (1.5,3.5) -- (8,3.5);
\filldraw (6.5,2.5) circle (.5ex);
\draw[line width = .2ex] (6.5,0) -- (6.5,2.5) -- (8,2.5);
\filldraw (0.5,1.5) circle (.5ex);
\draw[line width = .2ex] (0.5,0) -- (0.5,1.5) -- (8,1.5);
\filldraw (7.5,0.5) circle (.5ex);
\draw[line width = .2ex] (7.5,0) -- (7.5,0.5) -- (8,0.5);
\end{tikzpicture}};
\node[below left = 2 and 2 of c4] (c6) {\begin{tikzpicture}[scale=.4]
\draw (0,0) rectangle (8,8);

\draw (0,7) rectangle (1,8);
\draw (1,7) rectangle (2,8);
\draw (2,7) rectangle (3,8);
\draw (3,7) rectangle (4,8);
\draw (4,7) rectangle (5,8);
\draw (0,6) rectangle (1,7);
\draw (1,6) rectangle (2,7);
\draw (2,6) rectangle (3,7);
\draw (3,6) rectangle (4,7);
\draw (0,5) rectangle (1,6);
\draw (1,5) rectangle (2,6);
\draw (0,4) rectangle (1,5);
\draw (1,4) rectangle (2,5);
\draw (0,3) rectangle (1,4);
\draw (0,2) rectangle (1,3);

\node at (4.25,7.2) {X};

\filldraw (5.5,7.5) circle (.5ex);
\draw[line width = .2ex] (5.5,0) -- (5.5,7.5) -- (8,7.5);
\filldraw (4.5,6.5) circle (.5ex);
\draw[line width = .2ex] (4.5,0) -- (4.5,6.5) -- (8,6.5);
\filldraw (2.5,5.5) circle (.5ex);
\draw[line width = .2ex] (2.5,0) -- (2.5,5.5) -- (8,5.5);
\filldraw (3.5,4.5) circle (.5ex);
\draw[line width = .2ex] (3.5,0) -- (3.5,4.5) -- (8,4.5);
\filldraw (1.5,3.5) circle (.5ex);
\draw[line width = .2ex] (1.5,0) -- (1.5,3.5) -- (8,3.5);
\filldraw (6.5,2.5) circle (.5ex);
\draw[line width = .2ex] (6.5,0) -- (6.5,2.5) -- (8,2.5);
\filldraw (0.5,1.5) circle (.5ex);
\draw[line width = .2ex] (0.5,0) -- (0.5,1.5) -- (8,1.5);
\filldraw (7.5,0.5) circle (.5ex);
\draw[line width = .2ex] (7.5,0) -- (7.5,0.5) -- (8,0.5);
\end{tikzpicture}};
  \begin{scope}[nodes = {draw = none}]
    \path (root) edge[->] (c1)
    (root) edge[->] (c2)
    (root) edge[->] (c3)
    (c3) edge[->] (c4)
    (c4) edge[->] (c5)
    (c4) edge[->] (c6)
      ;
  \end{scope}
\end{tikzpicture}
\caption{$\mathcal{T}(w)$ for $w = 54278316$. The $a$ indicates the accessible box of each node. The $X$'s describe which boxes of the parent moved. From this tree, we compute $\#{\sf Red}(w)=730158$.}
\label{fig:transTree}
\end{figure}

If $u$ is vexillary we define $\lambda(u)$ graphically by pushing all boxes of $D(u)$ northwest along the diagonal that it sits until a partition shape is reached; see \cite[Section~3.2]{KMY}.
Concluding our running example, from Figure~\ref{fig:transTree} we have
\begin{align*}
\#{\sf Red}(54278316) & = f^{\lambda(54672318)}+f^{\lambda(56274318)}
+f^{\lambda(65342718)}+f^{\lambda(64532718)}\\
\ & =f^{4,3,3,3,1,1}+f^{4,4,3,2,1,1}+f^{5,4,2,2,1,1}+f^{5,3,3,2,1,1}\\
\ & = 730158.
\end{align*}

This result is a mild variation of \cite[Proposition~2.8.1]{Manivel} (cf.~\cite{LS:transition}) using the marching moves. We make no claim of originality.

\begin{theorem}[cf.~\cite{LS:transition,Manivel, Knutson.Yong}]
\label{thm:noorig}
$\#{\sf Red}(w)=\sum_{v\in {\mathcal L}(w)}f^{\lambda(v)}$.
\end{theorem}
\noindent
\emph{Proof:}  We follow \cite[Section~5.2]{AdveRobichauxYong}, which elaborates
on the notions from \cite{Knutson.Yong} in the case of Schubert polynomials
${\mathfrak S}_w$. We refer to \cite[Chapter 2]{Manivel} for background.

Let $(r,c)$ be the accessible box of $w\in S_n$ and set $k=w^{-1}(c)$.
Also let $w'=w\cdot (r,k)$. Transition gives this recurrence for the \emph{Schubert polynomials}:
\begin{equation}
\label{eqn:trans}
{\mathfrak S}_w=x_r {\mathfrak S}_{w'} +\sum_{w''} {\mathfrak S}_{w''},
\end{equation}
where the summation is over the children $w''$ of $w$ in ${\mathcal T}(w)$.

Let $1^N\times w\in S_{N+n}$ send $j\mapsto j$ for $1\leq j\leq n$ and $j\mapsto w(j-N+1)+N$
for $j\geq N$. Then
\[F_w=\lim_{N\to \infty} {\mathfrak S}_{1^N\times w}\in {\mathbb Z}[[x_1,x_2,\ldots \ ]].
\footnote{In the conventions of \cite{Stanley1984}, the limit is an expression for $F_{w^{-1}}$.}\]
Moreover, since $w\in {S}_n$ then
\begin{equation}
\label{eqn:Dec17abc}
F_w(x_1,x_2,\ldots, x_{n},0,0,\ldots)={\mathfrak S}_{1^{n}\times w}(x_1,x_2,\ldots,x_{n},0,0,\ldots).
\end{equation}
Now, by repeated application of (\ref{eqn:trans}) to $1^n\times w$,
\begin{equation}
\label{eqn:jun15abc}
{\mathfrak S}_{1^n\times w}=J(x_1,x_2,\ldots,x_{2n})+\sum_{v\in {\mathcal L}(w)}{\mathfrak S}_{1^n\times v},
\end{equation}
where $J(x_1,x_2,\ldots,x_n,0,0,\ldots)\equiv 0$. 

Hence by setting $x_i=0$ for $i>n$ in (\ref{eqn:jun15abc}) we obtain, using (\ref{eqn:Dec17abc}) that
\begin{equation}
\label{eqn:jun15cde}
F_w(x_1,\ldots,x_n)=\sum_{v\in {\mathcal L}(w)} F_v(x_1,\ldots,x_n).
\end{equation}
Let $s_{\alpha}(x_1,\ldots,x_n)$ be the Schur polynomial for a shape $\alpha$. Since $v\in {\mathcal L}(w)$ is vexillary, 
\[F_v(x_1,\ldots,x_n)=s_{\lambda(v)}(x_1,\ldots,x_n);\]
see, e.g., \cite[Section~2.8.1]{Manivel}. Hence
\begin{equation}
\label{eqn:jun15xyz}
F_w(x_1,\ldots,x_n)=\sum_{v\in {\mathcal L}(w)}s_{\lambda(v)}(x_1,\ldots,x_n).
\end{equation}
We have that
$[x_1 x_2\cdots x_{\ell(w)}]F_w=\#{\sf Red}(w)$
and
$[x_1 x_2\cdots x_{\ell(w)}]s_{\lambda(v)}(x_1,\ldots,x_n)=f^{\lambda(v)}$.
Now the result follows from these two facts combined with (\ref{eqn:jun15xyz}).\qed

\section{Proof of Theorems~\ref{thm:averagemain} and~\ref{thm:exponentialintro}} 
\subsection{On the distribution of ${\sf EG}(w)$}

\begin{lemma}
\label{lemma:basis}
For any $w\in {S}_n$, 
${\sf EG}(w)=\#{\mathcal L}(w)$.
\end{lemma}
\begin{proof}
Combining results of \cite{Stanley1984, Edelman.Greene} gives
\begin{equation}
\label{eqn:Nov20abc}
F_{w}(x_1,\ldots,x_{\ell(w)})=\sum_{\lambda} a_{w,\lambda}s_{\lambda}(x_1,\ldots,x_{\ell(w)})
\end{equation}
where the sum is over partitions $\lambda$ of size $\ell(w)$, and 
$a_{w,\lambda}$ is defined in Section~\ref{sec:1}.

The Schur polynomials $s_{\lambda}(x_1,\ldots,x_{\ell(w)})$ for $|\lambda|=\ell(w)$ are a basis of the vector space $\Lambda^{(\ell(w))}_{\mathbb Q}[x_1,\ldots,x_{\ell(w)}]$
of degree $\ell(w)$ symmetric polynomials in $\{x_1,\ldots,x_{\ell(w)}\}$. Since (\ref{eqn:Nov20abc}) and (\ref{eqn:jun15xyz}) (where $n=\ell(w)$) are
linear combinations for the same vector, we are done.
\end{proof}

In view of Lemma~\ref{lemma:basis}, Theorems~\ref{thm:averagemain} and~\ref{thm:exponentialintro} are equivalent. It is easy to see that
Theorem~\ref{thm:averagemain} follows from our main result:

\begin{theorem}
\label{thm:firstmain}
Fix $0<\gamma<\frac{1}{2}$. There exists $\alpha>0$ such that for $n$ sufficiently large,
\[{\mathbb P}({\sf EG}(w)\geq 2^{\alpha n})
\geq 1-\frac{1}{n^{2\gamma}}.\] 
\end{theorem}

\noindent\emph{Proof of Theorem~\ref{thm:firstmain}:}
Let $N_{\pi,n}(w)$ be the number of $\pi$ patterns
contained in $w\in S_n$.

\begin{proposition}
\label{claim:def}
Suppose in ${\mathcal T}(w)$ that the node $u$ has exactly one child $u'$. then 
\[N_{2143,n}(u')\geq N_{2143,n}(u).\]
\end{proposition}
\noindent
\emph{Proof of Proposition~\ref{claim:def}:}
Let the accessible box ${\bf z}_u$ of $u$ be in position $(x,y)$. By definition of $D(u)$, there is
a $\bullet$ of $G(u)$ at $C=(x,y')$ for some $y'>y$, and there is a $\bullet$ at $B=(x',y)$ for some $x'>x$. Let $b_1$ be the unique pivot of
$D(u)$, i.e., the southeastmost $\bullet$ that is northwest of ${\bf z}_u$
(as in Section~\ref{sec:2.2}). Suppose $b_1$ is at position $A=(c,d)$. Thus, $c<x$ and $d<y$. 

By definition of the transition algorithm,
all $\bullet$'s of $G(u)$ and $G(u')$ are in the same position, except $A,B,C$ in $G(u)$ are respectively replaced by $A',B',C'$ in $G(u')$ where
\begin{align*}
A=(c,d) & \mapsto  A'=(x,d)\\
B=(x',y) & \mapsto  B'=(c,y)\\
C=(x,y') & \mapsto  C'=(x',y')
\end{align*}

Schematically, the march move
looks as follows (we have thickened the moving $\bullet$'s).

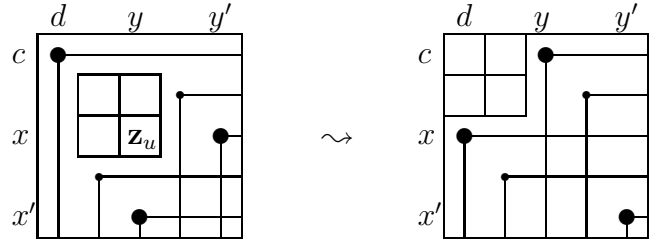
\setlength{\unitlength}{.27mm}
\[
\begin{picture}(320,115)
\put(0,15){\framebox(100,100)}
\put(10,105){\circle*{8}}
\put(10,105){\line(1,0){90}}
\put(10,105){\line(0,-1){90}}
\put(-13,101){$c$}
\put(6,119){$d$}
\put(44,119){$y$}
\put(84,119){$y'$}
\put(-13,61){$x$}
\put(-13,21){$x'$}
\put(44,61){${\bf z}_u$}
\put(140,60){$\leadsto$}
\put(70,85){\circle*{4}}
\put(70,85){\line(1,0){30}}
\put(70,85){\line(0,-1){70}}

\put(90,65){\circle*{8}}
\put(90,65){\line(1,0){10}}
\put(90,65){\line(0,-1){50}}

\put(30,45){\circle*{4}}
\put(30,45){\line(1,0){70}}
\put(30,45){\line(0,-1){30}}

\put(50,25){\circle*{8}}
\put(50,25){\line(1,0){50}}
\put(50,25){\line(0,-1){10}}

\thinlines

\put(20,55){\framebox(40,40)}
\put(20,75){\line(1,0){40}}
\put(40,55){\line(0,1){40}}


\put(200,15){\framebox(100,100)}
\put(250,105){\circle*{8}}
\put(250,105){\line(1,0){50}}
\put(250,105){\line(0,-1){90}}
\put(270,85){\circle*{4}}
\put(270,85){\line(1,0){30}}
\put(270,85){\line(0,-1){70}}

\put(187,101){$c$}
\put(206,119){$d$}
\put(244,119){$y$}
\put(284,119){$y'$}
\put(187,61){$x$}
\put(187,21){$x'$}

\put(210,65){\circle*{8}}
\put(210,65){\line(1,0){90}}
\put(210,65){\line(0,-1){50}}

\put(230,45){\circle*{4}}
\put(230,45){\line(1,0){70}}
\put(230,45){\line(0,-1){30}}

\put(290,25){\circle*{8}}
\put(290,25){\line(1,0){10}}
\put(290,25){\line(0,-1){10}}

\thinlines

\put(200,75){\framebox(40,40)}
\put(200,95){\line(1,0){40}}
\put(220,75){\line(0,1){40}}
\end{picture}
\]

\begin{claim} 
\label{claim:anotherjune1claim}
If there are two $\bullet$'s, other than $\{B,C\}$, that are weakly south and weakly east 
of ${\bf z}_u$ then one $\bullet$ must be (strictly) southeast of the other.
\end{claim}
\noindent
\emph{Proof of Claim~\ref{claim:anotherjune1claim}:} Suppose not. Then 
let the two $\bullet$'s be at $(q,r)$ and $(m,n)$ where $q>m$ and $r<n$. Then there is a box
${\bf z}\neq {\bf z}_u$ of $D(u)$ in position $(m,q)$, which is weakly south and weakly east of ${\bf z}_u$. 
Since ${\bf z}_u$ is not in the dominant component of $D(u)$, then
${\bf z}$ cannot be in that component either. Therefore, 
${\bf z}_u$ is not the accessible box of $D(u)$, a contradiction.
\qed

\begin{claim} 
\label{claim:june1claim}
There is no $\bullet$ of $G(u)$ strictly north of row $c$ and strictly between columns $d$ and $y$.
Similarly, there is no $\bullet$ of $G(u)$ strictly west of column $d$ and strictly between rows $c$ and $x$.

\end{claim}
\noindent
\emph{Proof of Claim~\ref{claim:june1claim}:} We prove only the first sentence of the claim, as the second sentence is analogous. Suppose not; we may assume this $\bullet$ is maximally
southeast with the assumed properties. Then $A=(c,d)$ and this $\bullet$
are two pivots for $D(u)$, which implies $u$ has at least two children, 
contradicting the hypothesis of the Proposition.
\qed

Let ${\mathcal F}_u$ consist of all embedding positions $i_1<i_2<i_3<i_4$ 
of a $2143$-pattern in $u$. Also, let ${\mathcal F}'_u$ be the subset of ${\mathcal F}_u$
consisting of those $i_1<i_2<i_3<i_4$ such that 
\[\{i_1,i_2,i_3,i_4\}\cap \{c,x,x'\}=\emptyset\] 
(i.e., the positions do not involve the rows of $A,B$ or $C$). Let 
\[{\mathcal F}''_u={\mathcal F}_u\setminus {\mathcal F}'_u.\]
Similarly, we define ${\mathcal F}_{u'}$, ${\mathcal F}'_{u'}$ and ${\mathcal F}''_{u'}$ in exactly the same way, except with respect to $u'$.

Since ${\mathcal F}'_u={\mathcal F}'_{u'}$, it suffices to establish an injection 
\[\psi:{\mathcal F}''_u\hookrightarrow {\mathcal F}''_{u'}.\]
In what follows, we will let $\bullet_i$ refer to the $\bullet$
in the diagram corresponding to the ``$i$'' in the $2143$ pattern, for $1\leq i\leq 4$. In addition,
if $i_1$ is in the row of $A$ we will write ``$A=\bullet_2$'', \emph{etc}.  
We define now $\psi$ in cases:

\smallskip
\noindent
{\sf Case 1:} ($B=\bullet_1$ or $B=\bullet_2$) The $\bullet_4$ and $\bullet_3$ appear strictly right of column $y$. This contradicts 
Claim~\ref{claim:anotherjune1claim}. Hence, no elements of ${\mathcal F}''_u$ fall into this case. 

\smallskip
\noindent
{\sf Case 2:} ($C=\bullet_1$ or $C=\bullet_2$) $\bullet_4$ and $\bullet_3$ appear strictly southeast of ${\bf z}_u$. As in {\sf Case 1}, this contradicts Claim~\ref{claim:anotherjune1claim}.
Again, no elements of ${\mathcal F}''_u$ fall into this case.

\smallskip\noindent
{\sf Case 3:} ($A=\bullet_1$) Let $\bullet_2$ be at position $(r,s)$. Hence $r<c$ and $s>d$.
If moreover, $s<y$ we contradict the first sentence of Claim~\ref{claim:june1claim}. Hence, $s>y$.  We must have that $\bullet_2\not\in \{A,B,C\}$ and $\bullet_4$ and $\bullet_3$ are strictly to the right of column $y$.

\noindent
{\sf Subcase 3a:} ($\bullet_4$ and $\bullet_3$ are both strictly south of row $x$) This contradicts Claim~\ref{claim:anotherjune1claim}.

\noindent
{\sf Subcase 3b:} ($\bullet_4$ and $\bullet_3$ are both strictly north of row $x$) 
Thus $\{\bullet_3,\bullet_4\}\cap \{A,B,C\}=\emptyset$. The $2143$ pattern
$[\bullet_2,A,\bullet_4,\bullet_3]$ is destroyed by the marching move,
i.e., $[\bullet_2,A',\bullet_4,\bullet_3]$ is not a $2143$ pattern in $u'$. Now, in $u'$ we now have the $2143$ pattern
$[\bullet_2,B',\bullet_4,\bullet_3]$. Hence we define
\[\psi([\bullet_2,A,\bullet_4,\bullet_3]):=[\bullet_2,B',\bullet_4,\bullet_3].\]

\noindent
{\sf Subcase 3c:} ($\bullet_4$ is strictly north of row $x$ and $\bullet_3$ is strictly south of row $x$).
Since $s>y$, $C\neq \bullet_3$. Hence $\{\bullet_3,\bullet_4\}\cap \{A,B,C\}=\emptyset$.
 The $2143$ pattern $[\bullet_2,A,\bullet_4,\bullet_3]$ is destroyed by the marching move. However, in $u'$ we now have the $2143$ pattern $[\bullet_2,B',\bullet_4,\bullet_3]$. We again  define
\[\psi([\bullet_2,A,\bullet_4,\bullet_3]):=[\bullet_2,B',\bullet_4,\bullet_3].\]

\noindent
{\sf Subcase 3d:} ($\bullet_3$ is in row $x$ and $\bullet_4$ is strictly above row $x$) Then in fact $\bullet_3=C$
while $\bullet_4\not\in \{A,B,C\}$. In this case,
\[\psi([\bullet_2,A,\bullet_4,C]):=[\bullet_2, B', \bullet_4,C'].\]

\noindent
{\sf Subcase 3e:} ($\bullet_4$ is in row $x$ and $\bullet_3$ is strictly south of row $x$) Thus $\bullet_4=C$ and $\bullet_3$
is strictly southeast of $z_u$. This contradicts Claim~\ref{claim:anotherjune1claim}.

\smallskip\noindent
{\sf Case 4:} ($A=\bullet_2$) Let the $1$ be at position $(r,s)$. Hence $r>c$ and $s<d$. If $r\leq x$ then we contradict the second sentence of Claim~\ref{claim:june1claim}. Hence $r>x$. We have that $\bullet_4$ and $\bullet_3$ are in rows strictly south of $x$. Moreover, there must be a box $e$ of $D(u)$ in the row of $\bullet_4$ and the column of $\bullet_3$ that is therefore strictly south of ${\bf z}_u$. Since the columns of $\bullet_4$ and $\bullet_3$ are strictly east of column $d$, the box $e$ is not part of the dominant component of $D(u)$. Hence, ${\bf z}_u$ cannot be the accessible box, a contradiction. Thus, no elements of ${\mathcal F}''_u$ are in this case.

\smallskip
\noindent
{\sf Case 5:} ($A=\bullet_3$) Hence, in $u$, $\bullet_{2},\bullet_{1},\bullet_{4}$ are strictly north of row $c$.
Thus $\{\bullet_1,\bullet_2,\bullet_4\}\cap \{A,B,C\}=\emptyset$ and
 $\bullet_{2},\bullet_{1},\bullet_{4}$ remain in the same place in $u'$. Set
\[\psi([\bullet_{2},\bullet_1,\bullet_4,A]):=[\bullet_2,\bullet_1,\bullet_4,A'].\]

\smallskip
\noindent
{\sf Case 6:} ($A=\bullet_4$) $\bullet_3$ is strictly south of the row of $A$. If it is also weakly north
of $x$, we contradict the second sentence of Claim~\ref{claim:june1claim}. Hence $\bullet_3$
is strictly south of $x$, i.e., the row of~$e$. Now, $\bullet_2,\bullet_1,\bullet_3$ are the same position
in $u$ and $u'$ and $\{\bullet_1,\bullet_2,\bullet_3\}\cap \{A,B,C\}=\emptyset$. Here,
\[\psi([\bullet_2,\bullet_1,A,\bullet_3]):=[\bullet_2,\bullet_1,A',\bullet_3].\]
Since the row of $A'$ is $x$ the output is a $2143$ pattern in $u'$.

\smallskip
\noindent
{\sf Case 7:} ($B=\bullet_4$) Let $\bullet_3$ be at
$(r,s)$. Thus $r>x'$ and $s<y$. There must be a
box $e\in D(w)$ in position $(x',s)$. Now, $\bullet_2$ and $\bullet_1$ are in columns strictly left of $s$ and strictly above row $r$. Hence $e$ cannot be in the dominant component
of $D(w)$. Thus, since $e$ is further south than ${\bf z}_u$, the latter is not  accessible, a contradiction. So, no elements of ${\mathcal F}''_u$ appear in this case.

\smallskip\noindent
{\sf Case 8:} ($B=\bullet_3$) Let $\bullet_4$
be in position $(r,s)$. 

\noindent
{\sf Subcase 8a:} ($r<c$) Therefore, $\bullet_1$ and $\bullet_2$ 
are also strictly above row $c$. Since $\bullet_1,\bullet_2$ and $\bullet_4$ stay in the same
place in $u$ and $u'$ and $B'$ is in row $c$ in $u'$. Moreover, $\{\bullet_1,\bullet_2,\bullet_4\}\cap \{A,B,C\}=\emptyset$. We may define
\[\psi([\bullet_2,\bullet_1,\bullet_4,B]):=[\bullet_2,\bullet_1,\bullet_4,B'].\]

\noindent
{\sf Subcase 8b:} ($x<r<x'$) This contradicts Claim~\ref{claim:anotherjune1claim}.

\noindent
{\sf Subcase 8c:} $(r=c)$ This implies $A=\bullet_4$, which is impossible.

\noindent
{\sf Subcase 8d:} ($c<r< x$) We may assume $A\neq \bullet_1,\bullet_2$ since those cases are handled by
{\sf Case 3} and {\sf Case 4}. Now, $\bullet_1$ and $\bullet_2$ are strictly west of
column $y$ and strictly north of row $x$. By the assumption that $A=b_1$ is the (unique) pivot, combined with Claim~\ref{claim:june1claim},
both $\bullet_1$ and $\bullet_2$ are strictly northwest of $A$. Thus,
$\bullet_1$ and $\bullet_2$ are in the same place
in $u'$, and $\{\bullet_1,\bullet_2,\bullet_4\}\cap \{A,B,C\}=\emptyset$.
Since $A'$ is in row $x$, it make sense to let
\[\psi([\bullet_2,\bullet_1,\bullet_4,B]):=[\bullet_2,\bullet_1,\bullet_4,A'].\]

\noindent
{\sf Subcase 8e:} $(r=x)$ Hence $C=\bullet_4$. For the same reasons as in {\sf Subcase 8d},
both $\bullet_1$ and $\bullet_2$ are strictly northwest of $A$. Thus,
$\bullet_1$ and $\bullet_2$ are in the same place
in $u'$ and $\{\bullet_1,\bullet_2\}\cap \{A,B,C\}=\emptyset$. In this case set
\[\psi([\bullet_2,\bullet_1,C,B]):=[\bullet_2,\bullet_1,B',A'].\]

\smallskip
\noindent
{\sf Case 9:} ($C=\bullet_3$) Let $\bullet_4$ be in position $(r,s)$. Hence $s>y'$.

\noindent
{\sf Subcase 9a:} ($r<c$) Hence $\bullet_1,\bullet_2$ and $\bullet_4$ remain in the same place
in $u'$ and $\{\bullet_1,\bullet_2,\bullet_4\}\cap \{A,B,C\}=\emptyset$. Since $C'$ is further south than $C$, we may set
\[\psi([\bullet_2,\bullet_1,\bullet_4,C]):=[\bullet_2,\bullet_1,\bullet_4,C'].\]

\noindent
{\sf Subcase 9b:} ($c<r<x$) We may also assume that $A\neq\bullet_1$ and $A\neq \bullet_2$, since those are handled in {\sf Case 3} and {\sf Case 4}, respectively. Thus $\{\bullet_1,\bullet_2,\bullet_4\}\cap \{A,B,C\}=\emptyset$. Here, 
\[\psi([\bullet_2,\bullet_1,\bullet_4,C]):=[\bullet_2,\bullet_1,\bullet_4,C'].\]

\noindent
{\sf Subcase 9c:} $(r=c)$ Then $A=\bullet_4$, which is impossible.

\smallskip
\noindent
{\sf Case 10:} ($C=\bullet_4$) We may assume that $A\neq \bullet_1,\bullet_2$ ({\sf Case 3} and {\sf Case 4}) and
also $B\neq \bullet_3$ ({\sf Case 8}). Therefore $\{\bullet_1,\bullet_2,\bullet_3\}\cap \{A,B,C\}=\emptyset$.
Let $\bullet_3$ be in position $(r,s)$. 

\noindent
{\sf Subcase 10a:} ($y<s<y'$) This contradicts Claim~\ref{claim:anotherjune1claim}.

\noindent
{\sf Subcase 10b:} ($s=y$) This means $\bullet_3=B$, a situation we have ruled out/refer to {\sf Case 8}.

\noindent
{\sf Subcase 10c:} ($s<y$) If moreover
$r>x'$ then there exists $e\in D(w)$ in position $(x',s)$, which is therefore strictly south of ${\bf z}_u$. Since column $s$ is strictly east of the column of $\bullet_2$, $e$ is not in the dominant component. Hence ${\bf z}_u$ is not accessible, a contradiction.
Now $r\neq x'$ (since we assumed $B\neq \bullet_3$). Thus, $x<r<x'$
and it follows that $\bullet_3$ is in the same place in $u'$. By the reasoning of the first paragraph of
{\sf Subcase 8d}, $\bullet_1,\bullet_2$ are strictly northwest of $A$. Hence $\bullet_1,\bullet_2$ also remain in the same
place in $u'$. Summing up, since  $B'$ is in row $c$, we may define
\[\psi([\bullet_2,\bullet_1,C,\bullet_3]):=[\bullet_2,\bullet_1,B',\bullet_3].\]
\noindent
\emph{$\psi$ is well-defined:} The above cases handle each of the possibilities
for $A,B,C$ being one of $1,2,3,4$. Our definition of $\psi$ is shown 
to send an element of ${\mathcal F}''_u$ to an element of ${\mathcal F}''_{u'}$.

We also need that if an element of ${\mathcal F}''_u$ occurs in two cases, 
$\psi$ sends them to the \emph{same} element of  ${\mathcal F}''_{u'}$.
By inspection, the only overlapping situations are 
{\sf Subcase 3d}$\leftrightarrow${\sf Subcase 9b} and
{\sf Subcase 8d}$\leftrightarrow${\sf Case 10}. In both these cases we
define $\psi$ to be consistent on the overlap. 

\noindent
\emph{$\psi$ is an injection:} This is by inspection of pairs of subcases where $\psi$'s output was given.
By our choice of notation, if $\bullet_i$ appears in the description of the input to $\psi$, 
it cannot be equal to $A,B$ or $C$ and hence in the output, it cannot be equal to $A',B'$ or $C'$ (as $\{A,B,C\}$ and $\{A',B',C'\}$ occupy the same rows).
Therefore, if in two cases, some coordinate of the two outputs differ \emph{symbolically}, those outputs cannot be equal.
After ruling out these pairs, we are left with a few to check:

\noindent
{\sf Subcase 3b}, {\sf Subcase 3c}: These differ in the fourth coordinate since in the former
case, $\bullet_3$ is strictly north of row $x$ and in the latter case, $\bullet_3$ is strictly south of row $x$.

\noindent
{\sf Case 5} and {\sf Subcase 8d}: These differ in the third coordinate since in the former case, $\bullet_4$
appears above row $c$ whereas in the latter case, $\bullet_4$ is below row $c$.

\noindent
{\sf Subcase 9a} and {\sf Subcase 9b}: These differ in the third coordinate for the same reason as
the previous pair.\qed

\begin{lemma}
\label{claim:abc}
Let $w\in {S}_n$ and suppose $u\to u'$ in ${\mathcal T}(w)$. Then
\[N_{2143,n}(u)-N_{2143,n}(u')\leq 2n^3+3n^2-n.\]
\end{lemma}
\noindent
\emph{Proof of Lemma~\ref{claim:abc}:}
Since $u\to u'$ in ${\mathcal T}(w)$, exactly three positions $a,i,j$ differ between $u$ and $u'$. We are claiming that
\begin{equation}
\label{eqn:Nov27abcd}
N_{2143,n}(u)-N_{2143,n}(u')\leq {3\choose 3}4{n\choose 1}+{3\choose 2}6{n\choose 2}+{3\choose 1}4{n\choose 3}=
2n^3+3n^2-n.
\end{equation} 
Let $t_1<t_2<t_3<t_4$ be the indices of a $2143$-pattern in $u$. First suppose $\{t_1,t_2,t_3,t_4\}\cap \{a,i,j\}=\emptyset$. Clearly, $t_1<t_2<t_3<t_4$
are indices of a $2143$-pattern in $u'$. Therefore this case does not contribute to $N_{2143,n}(u)-N_{2143,n}(u')$.

Next assume $\#(\{t_1,t_2,t_3,t_4\}\cap \{a,i,j\})=1$. There
are ${3\choose 1}$ choices for which of $a,i$ or $j$ is in $\{t_1,t_2,t_3,t_4\}$. Then there are at most ${n \choose 3}$ choices for $\{t_1,t_2,t_3,t_4\}\setminus \{a,i,j\}$.
Finally there are $4$ choices for which $k$ satisfies $t_k\in \{a,i,j\}$. Therefore, this case contributes at most ${3\choose 1}4{n\choose 3}$ to $N_{2143,n}(u)-N_{2143,n}(u')$, thus
explaining the third term of (\ref{eqn:Nov27abcd}).
 
Similar arguments explain the first and second terms of (\ref{eqn:Nov27abcd}) as the contributions to $N_{2143,n}(u)-N_{2143,n}(u')$ from the
cases that 
\[\#(\{t_1,t_2,t_3,t_4\}\cap \{a,i,j\})=3 \text{\  and \ 
$\#(\{t_1,t_2,t_3,t_4\}\cap \{a,i,j\})=2$},\] 
respectively. The lemma thus follows.
\qed

The following is known; see work of M.~Bona \cite{Bona} and 
of S.~Janson, B.~Nakamura, and D.~Zeilberger \cite{Janson}. The proof being
not difficult, we include it for completeness.

\begin{lemma}
\label{prop:Bona}
For any $\pi\in S_k$, the expected
number of occurrences of $\pi$ as a pattern in $w\in {S}_n$ (selected using the uniform distribution) is
${n\choose k}/k!$.
\end{lemma}
\noindent\emph{Proof of Lemma~\ref{prop:Bona}:}
For an increasing sequence 
$I=\{i_1<i_2<\ldots<i_k\}$ (in $[1,n]$), let
\begin{equation}
\label{eqn:nowlet}
{\sf X}_{I}(w)=
\begin{cases}
1 \text{ \ if $\pi$ is a pattern at the positions of $I$;}\\
0 \text{ \ otherwise.}
\end{cases}
\end{equation}
Thus, $N_{\pi,n}=\sum_{I} {\sf X}_I$.
There are ${n\choose k}(n-k)!$
permutations such that $I$ has pattern $\pi$. By linearity of expectation,
\[\mathbb{E}[N_{\pi,n}]=\sum_I \mathbb{E}[{\sf X}_I] ={n\choose k}^2 (n-k)!/n!,\]
and the lemma follows.\qed 

\begin{lemma}
\label{lemma:xyz}
Let ${\mathcal T}$ be a
rooted tree with the property that along any path from the root to  a leaf 
there are $d$ nodes with at least two children. Then that tree has at least $2^d$ leaves.
\end{lemma}
\noindent
\emph{Proof of Lemma~\ref{lemma:xyz}:}
Arbitrarily left-right order the descendants of the root of ${\mathcal T}$. After pruning, if necessary, we may assume each node as at most two children.
Along any path from the root to a leaf, record ``$S$" if a node has one child, and ``$L$'' if one steps to the left child and ``$R$'' if one goes to the right child. Thus, each leaf is uniquely encoded by an $\{S,L,R\}$ sequence. By hypothesis, each such sequence has at least $d$ from $\{L,R\}$. Also, each of the $2^d$-many $\{L,R\}$-sequences must be a subsequence of a unique leaf sequence. Hence there are at least $2^d$ leaves. \qed

By Chebyshev's inequality, for any $t\in {\mathbb R}_{>0}$,
${\mathbb P}(|N_{\pi,n}-\mu|\geq t\sigma)\leq 1/t^2$,
and hence
\[{\mathbb P}(N_{\pi,n}\geq \mu-t\sigma)\geq 1-1/t^2.\]
For $\pi=2143$, $\mu={n\choose 4}/4!$. Let $t=n^\gamma$ for the fixed choice $0<\gamma<\frac{1}{2}$. Thus, we obtain
\begin{equation}
\label{eqn:Dec3abc}
{\mathbb P}\left(\frac{N_{2143,n}}{2n^3+3n^2+n}\geq \frac{{n\choose 4}/4!-n^{\gamma}\sigma}{2n^3+3n^2+n}\right)\geq 1-
\frac{1}{n^{2\gamma}}.
\end{equation}
Define a random variable $Q:S_n\to {\mathbb Z}_{\geq 0}$ by 
\[Q(w)=\min_{u\in {\mathcal L}(w)} \#\{v \text{ appears in a path from $w$ to $u$ in ${\mathcal T}(w)$}: \exists v'\neq v'', v\to v', v\to v''\}.\] 
By Proposition~\ref{claim:def} and Lemma~\ref{claim:abc},
\begin{equation}
\label{eqn:Dec3cde}
Q(w)\geq \frac{N_{2143,n}(w)}{2n^3+3n^2+n}.
\end{equation}
Combining (\ref{eqn:Dec3abc}) and (\ref{eqn:Dec3cde}) gives
\[
{\mathbb P}\left(Q\geq \frac{{n\choose 4}/4!-n^{\gamma}\sigma}{2n^3+3n^2+n}\right)\geq 1-
\frac{1}{n^{2\gamma}}.
\]

By \cite[Section~4.1]{Janson},
the $r$-th central moment for $N_{\pi,n}$, i.e., 
${\mathbb E}[(N_{\pi,n}-{\mathbb E}(N_{\pi,n}))^r]$, is a polynomial in $n$ of
degree $\lfloor r(k-\frac{1}{2})\rfloor$ where, recall, $\pi\in S_k$. Hence 
${\sf Var}(N_{2143,n}) \in O(n^7)$ and $\sigma\in O(n^{3.5})$. 
Therefore there exists $\alpha>0$ such that for $n$ sufficiently large
\begin{equation}
\label{eqn:def}
{\mathbb P}(Q\geq \alpha n)\geq 1-\frac{1}{n^{2\gamma}}.
\end{equation}

Finally, by Lemma~\ref{lemma:basis} and Lemma~\ref{lemma:xyz}, 
\begin{equation}
\label{eqn:ghi}
{\sf EG}(w)=\#{\mathcal L}(w)\geq 2^{Q(w)}.
\end{equation}
The desired equality holds by (\ref{eqn:def}) and (\ref{eqn:ghi}) combined.\qed

\subsection{Remarks}
M.~Bona \cite{Bona} proves that the sequence of random variables 
\[\widetilde{X_n}:=\frac{N_{2143,n}-{\mathbb E}[N_{2143,n}]}{\sqrt{{\sf Var}(N_{2143,n})}}\]
is asymptotically normal, i.e., $X_n$ converges in distribution to the standard normal variable $N(0,1)$.
In particular, this means that for any $\epsilon>0$, for any $a,b\in {\mathbb R}$, there exists $N\in {\mathbb N}$
such that for all $n\geq N$,
$|{\mathbb P}({\widetilde X}_n\in [a,b])-{\mathbb P}(N(0,1)\in [a,b])|<\epsilon$.
Thus one could use Bona's theorem to prove a more refined version of Theorem~\ref{thm:firstmain}. However, this does not
affect our basic conclusions, so 
 we opted to state a result/proof that only appeals to Chebyshev's inequality.

In \cite{Billey},  $w\in {S}_n$ is defined to be \emph{$k$-vexillary} if ${\sf EG}(w)=k$. I.~G.~Macdonald \cite{Macdonald} proves that the proportion of vexillary permutations in ${S}_n$ goes to zero as $n\to \infty$. Extending this,
Theorem~\ref{thm:firstmain} implies:
\begin{corollary}
\label{cor:macgen}
Fix a positive integer $k$. Then
$\lim_{n\to\infty}{\mathbb P}(w\in {S}_n \text{\ is $k$-vexillary})\to 0$.
\end{corollary}

Using the relations
\begin{equation}
\label{eqn:therelations}
s_{i}s_j=s_j s_i \text{ for $|i-j|\geq 2$, and \ } s_i s_{i+1} s_i =s_{i+1}s_i s_{i+1}
\end{equation}
one can transform between any two reduced words
\[s_{i_1}s_{i_2}\cdots s_{i_{\ell}} \iff s_{j_1}s_{j_2}\cdots s_{j_{\ell}}\in {\sf Red}(w);\]
see, e.g., \cite[Proposition~2.1.6]{Manivel}. Hence, it follows that 
\begin{equation}
\label{eqn:support}
\{i_1,i_2,\ldots,i_{\ell}\}=\{j_1,j_2,\ldots,j_{\ell}\}.
\end{equation}

Let $\sigma^{(n)}=214365\cdots 2n \ 2n-1\in {S}_{2n}$.
\begin{proposition}
\label{lemma:june6}
$a_{\sigma^{(n)},\lambda}=f^{\lambda}$.
\end{proposition}
\begin{proof}
Fix any partition $\lambda$ 
of size $2n-1$. Consider any row and column increasing filling $T$ of $\lambda$, using each of the labels $\{1,3,5,\ldots,2n-1\}$ precisely once. Let ${\mathcal A}_{\lambda}$ be the set of these tableaux.
Also, let ${\mathcal B}_{\lambda}$ be the set of EG tableaux for the coefficient
$a_{\sigma^{(n)},\lambda}$. ${\sf Red}(\sigma^{(n)})$ consists of all $n!$ rearrangements of the factors of
$s_1 s_3\cdots s_{2n-1}$.  Hence, the column reading word of any $T\in {\mathcal A}_{\lambda}$ gives a reduced word for $w$. Thus, ${\mathcal A}_{\lambda}\subseteq {\mathcal B}_{\lambda}$.
By (\ref{eqn:support}), if $S\in {\mathcal B}_{\lambda}$, it must use each label of
$\{1,3,5,\ldots,2n-1\}$ exactly once. Since $S$ must also be row and column increasing, we see
$S\in {\mathcal A}_{\lambda}$. This gives ${\mathcal A}_{\lambda}={\mathcal B}_{\lambda}$.

Given $T\in {\mathcal A}_{\lambda}(={\mathcal B}_{\lambda})$, let $\phi(T)\in {\sf SYT}(\lambda)$ be the standard Young tableau of shape $\lambda$ obtained by sending label $i$ in $T$ to $\left\lceil\frac{i}{2}\right\rceil$. Clearly, $\phi:{\mathcal A}_{\lambda}\to {\sf SYT}(\lambda)$ is a
bijection. Hence $a_{\sigma^{(n)},\lambda}=\#{\mathcal A}_{\lambda}=\#{\sf SYT}(\lambda)=f^{\lambda}$.
\end{proof}

Let ${\sf inv}(n)$ be the number of involutions of ${S}_n$. The following shows that
the worst case and average case running time of transition is quite different:

\begin{corollary}
\label{cor:whatwasI}
$\#{\mathcal L}(\sigma^{(n)})={\sf inv}(n)
\sim \left(\frac{n}{e}\right)^{n/2}
\frac{e^{\sqrt{n}}}{(4e)^{\frac{1}{4}}}
$. 
\end{corollary}
\begin{proof}
The equality holds since
\begin{equation}
\label{eqn:412xyz}
\#{\mathcal L}(\sigma^{(n)})={\sf EG}(\sigma^{(n)})=\sum_{\lambda} f^{\lambda}={\sf inv}(n).
\end{equation}
The first equality of (\ref{eqn:412xyz}) is Lemma~\ref{lemma:basis}, 
the second is Proposition~\ref{lemma:june6} and the third is textbook (e.g., \cite[Corollary~7.13.9]{ECII}). The asymptotic statement is
\cite[Section~5.1.4]{Knuth:art}.
\end{proof}

The following conjecture has been proved by G.~Orelowitz (private communication):
\begin{conjecture}
$a_{w,\lambda}\leq f^{\lambda}$.
\end{conjecture}
We refer to his paper (in preparation) for application to the Edelman-Greene statistic.

\section{Counting Hecke words}
\label{sec:4}
A sequence $(i_1,i_2,\ldots,i_N)$ is a \emph{Hecke word} for 
$w\in {S}_n$ if $s_{i_1}\star s_{i_2}\star\cdots \star s_{i_N}=w$ where $\star$ is the
\emph{Demazure product} defined by 
\[u\star s_i=\begin{cases}
us_i \text{ \ if $\ell(us_i)=\ell(u)+1$}\\
u \text{\ otherwise.}
\end{cases}\]

Therefore, $N\geq \ell(w)$. Let ${\sf Hecke}(w,N)$ denote the set of Hecke words for $w$ of length $N$. 

\subsection{Two generalizations of the Edelman-Greene formula (\ref{eqn:stanleyformula})}
We now give two formulas for computing ${\sf Hecke}(w,N)$. Both are known to experts, but we are unaware of any specific place that
they appear in the literature.

Since ${\sf Hecke}(w,\ell(w))={\sf Red}(w)$, formula (\ref{eqn:Grothflambda}) below generalizes (\ref{eqn:stanleyformula}). Our second point is that in contrast with (\ref{eqn:whenvex}), even for vexillary permutations,
(\ref{eqn:Grothflambda}) is not short.

\begin{proposition}
\label{prop:usefulChan}
There is a manifestly nonnegative combinatorial formula
\begin{equation}
\label{eqn:Grothflambda}
\#{\sf Hecke}(w,N)=\sum_{\lambda, |\lambda|=N} b_{w,\lambda}f^{\lambda},
\end{equation}
where $b_{w,\lambda}$ counts the number of row strictly increasing and column weakly
increasing tableaux of
shape $\lambda$ whose top to bottom, right to left, column reading word is a Hecke word for $w$.

Let $M\geq 1$. There is a vexillary permutation $\pi\in S_{2M}$ with $\ell(\pi)=M^2$ such that
\[\#\{\lambda\in {\sf par}(M^2+M):b_{\pi,\lambda}>0\}\geq  {\sf par}(M),\] 
where ${\sf par}(M)$ is the number of 
partitions of size $M$. That is when $w=\pi$ and $N=M^2+M$, (\ref{eqn:Grothflambda}) has at least ${\sf par}(M)$-many terms.
Moreover,
\[\sum_{\lambda:|\lambda|=M^2+M}b_{\pi,\lambda}\geq {\sf inv}(M).\]
\end{proposition}
\begin{proof}
We use the results of S.~Fomin-A.~N.~Kirillov \cite{Fomin.Kirillov} who prove the following combinatorial formula for the \emph{stable Grothendieck polynomial} $G_w$:
\[G_w=\sum_{({\bf i},{\bf j})}(-1)^{\ell(w)-|{\bf j}|}{\bf x}^{\bf j},\]
where ${\bf i}=(i_1,\ldots,i_N)\in {\sf Hecke}(w,N)$, and ${\bf j}=(j_1\leq j_2\leq\cdots \leq j_N)$ are positive integers satisfying $j_t<j_{t+1}$ whenever $i_t\leq i_{t+1}$.
This is a formal power series in $x_1,x_2,\ldots$. 

For any ${\bf i}={\sf Hecke}(w,N)$, the sequence $(1,2,\ldots,N)$ can be used for ${\bf j}$. Hence,
\begin{equation}
\label{eqn:408a}
(-1)^{N-\ell(w)}\#{\sf Hecke}(w,N)=[x_1 x_2\cdots x_N] G_w.
\end{equation}
S.~Fomin-C.~Greene \cite[Theorem~1.2]{Fomin.Greene} states that, up to change of
conventions,
\begin{equation}
\label{eqn:408b}
G_w=\sum_{\lambda}(-1)^{|\lambda|-\ell(w)}b_{w,\lambda}s_{\lambda}.
\end{equation}
Combining (\ref{eqn:408a}) and (\ref{eqn:408b}) 
gives (\ref{eqn:Grothflambda}).

C.~Lenart \cite{Lenart} gave an expression for the \emph{symmetric Grothendieck polynomial}:
\begin{equation}
\label{eqn:lenart}
G_{\mu}(x_1,x_2,\ldots,x_t)=\sum_{\lambda} (-1)^{|\lambda|-|\mu|}g_{\mu,\lambda}s_{\lambda}(x_1,\ldots,x_t)
\end{equation}
where $\mu\subseteq\lambda\subseteq \widehat{\mu}$. Here $\widehat{\mu}$ is the unique
maximal partition with $t$ rows obtained by adding at most $i-1$ boxes to row $i$ of $\mu$ for $2\leq i\leq t$. In addition, $g_{\mu,\lambda}$ counts the number of \emph{Lenart tableaux}, i.e., column and row strict
tableaux of shape $\mu/\lambda$ with entries in the $i$-th row restricted to $1,2,\ldots,i-1$
for each $i$. 

Pick $\mu=M\times M$ and fix $t\geq M^2+M$. Therefore
$\widehat\mu=t\times M$. Hence, by (\ref{eqn:lenart}),
\begin{equation}
\label{eqn:lenspec}
(-1)^{M}[x_1\cdots x_{M^2+M}]G_{M\times M}(x_1,\ldots,x_{t})=\sum_{\lambda} g_{M\times M,\lambda}f^{\lambda}.
\end{equation}
Here the sum is over $\mu\subseteq \lambda\subseteq {\widehat\mu}$ with $|\lambda|=M^2+M$. Now, each such $\lambda$ is of the form $(M\times M,\overline\lambda)$ where $\overline\lambda\in {\sf par}(M)$ and is
contained in $M\times M$. Notice that $g_{M\times M,\lambda}\geq f^{\overline \lambda}$, since for each such $\lambda$
we can obtain a Lenart-tableau by filling the $\overline\lambda$ part with $1,2,\ldots,M$
to obtain a standard tableau, in all possible ways. Hence, using \cite[Corollary~7.13.9]{ECII}, 
\[\sum_{\lambda\in {\sf Par}(M^2+M)}g_{M\times M,\lambda}\geq \sum_{\overline\lambda:|\overline\lambda|=M} f^{\overline\lambda}={\sf inv}(M).\]

Finally, let $\pi=M+1,M+2,\ldots,2M,1,2,3\ldots,M\in S_{2M}$. This is a vexillary
permutation $\pi$ with
$\lambda(\pi)=M\times M$. By, e.g., \cite[Lemma~5.4]{KMY},
\[G_\pi(x_1,\ldots,x_t,0,0,\ldots,)=G_{M\times M}(x_1,\ldots,x_t,0,0,\ldots).\]
Since the Schur polynomials form a basis of the ring of symmetric polynomials, 
the righthand sides of (\ref{eqn:lenart}) and (\ref{eqn:408b})
coincide, i.e., 
$b_{\pi,\lambda}=g_{M\times M,\lambda}$ for every $\lambda$. The result follows.
\end{proof}

\begin{example}
\label{exa:Dec22}
Let $w=31524=s_4 s_2 s_3 s_1$. Using
(\ref{eqn:Grothflambda}) we obtain
\begin{align*}
\#{\sf Hecke}(w,5) & = \left(\tableau{1 & 2 & 4\\ 1 & 3}, \tableau{1 & 2 & 4\\ 3 & 4}\right)\! f^{3,2}
\!+\!\left(\tableau{1 & 2 & 4\\ 3 \\ 3}, \tableau{1 & 2 & 4\\ 1 \\ 3}\right)\! f^{3,1,1}\!+\!
\left(\tableau{1 & 2\\ 3 & 4\\ 3},\tableau{1 & 2\\ 1& 4\\ 3}\right)\! f^{2,2,1}
\\
& = 2f^{3,2}+2f^{3,1,1}+2f^{2,2,1}=32,
\end{align*}
which the reader may confirm by direct check. \qed
\end{example}

This next generalization of (\ref{eqn:stanleyformula}) is also manifestly nonnegative. It specializes in the vexillary case in a tantalizing way.

\begin{proposition}
\label{prop:secondG}
\begin{equation}
\label{eqn:Nov22abc}
\#{\sf Hecke}(w,N)=\sum_{\lambda: \ell(w)\leq|\lambda|\leq N} c_{w,\lambda} f^{\lambda,N},
\end{equation}
where $c_{w,\lambda}$ is the number of row and column strict tableaux of shape $\lambda$ whose top to bottom, right to left, column reading word is a Hecke word for $w$. If $w$ is vexillary, then 
\[\#{\sf Hecke}(w,N)=f^{\lambda(w),N}.\]
\end{proposition}
\begin{proof}
Work of A.~Buch, A.~Kresch, M.~Shimozono, H.~Tamvakis and the third author \cite{BKSTY} proves that 
\begin{equation}
\label{eqn:408c}
G_w=\sum_{\lambda} (-1)^{\ell(w)-|\lambda|}c_{w,\lambda}G_{\lambda}
\text{\ \ where \ \ 
$G_{\lambda}=\sum_T (-1)^{|T|-|\lambda|}{\bf x}^T$}
\end{equation}
and the latter sum is over all
semistandard set-valued tableaux of shape $\lambda$ \cite{Buch:KLR}. 
Therefore by (\ref{eqn:408a}) we have
\begin{align*}
(-1)^{N-\ell(w)} \#{\sf Hecke}(w,N) & =  [x_1 x_2\cdots x_N] G_w\\
\ & =  [x_1 x_2\cdots x_N] \sum_{\lambda} (-1)^{\ell(w)-|\lambda|}c_{w,\lambda} G_{\lambda}\\
\ & = \sum_{\lambda: |\lambda|\leq N} (-1)^{\ell(w)-|\lambda|}c_{w,\lambda}[x_1 x_2\cdots x_N] G_{\lambda}\\
\ & = \sum_{\lambda:|\lambda|\leq N}(-1)^{\ell(w)-|\lambda|}c_{w,\lambda}(-1)^{N-|\lambda|}f^{\lambda,N}\\
\ & = \sum_{\lambda:|\lambda|\leq N} (-1)^{N+\ell(w)}c_{w,\lambda}f^{\lambda,N},
\end{align*}
proving (\ref{eqn:Nov22abc}). For the second statement, by \cite[Lemma~5.4]{ReinerTennerYong},
when $w$ is vexillary then 
$G_w=G_{\lambda}$, and the above sequence of equalities simplifies, as desired.
\end{proof}

\begin{example}
Again let $w=31524$ as in Example~\ref{exa:Dec22}. Now applying (\ref{eqn:Nov22abc}) gives
\begin{align*}
\#{\sf Hecke}(w,5)  & = \left(\tableau{1 & 2\\ 3 &4}\right)f^{(2,2),5}+
\left(\tableau{1 & 2 & 4\\ 3}\right)f^{(3,1),5}+\left(\tableau{1 & 2 & 4\\ 3 & 4}\right)f^{(3,2),5}\\
 & = 10 + 17 + 5 = 32,
\end{align*}
in agreement with Example~\ref{exa:Dec22}. One can check the $f^{\lambda,N}$ computations either directly, or by using 
\begin{equation}
\label{eqn:dec22vvv}
f^{\lambda,N}=[x_1\cdots x_N]G_{\lambda}
\end{equation}
combined with (\ref{eqn:lenart}).\qed
\end{example}

Proposition~\ref{prop:secondG} is our central motivation for Problem~\ref{conj:Dec16conj}.

\begin{proposition}
\label{prop:plus2}
Fix $k$. There is an $|\mu|^{O(1)}$ algorithm to compute
 $f^{\mu,N}$ where $N\leq |\mu|+k$.
\end{proposition}
\begin{proof}
We use (\ref{eqn:dec22vvv}) combined with (\ref{eqn:lenart}) and
describe the possible Lenart tableaux. First, we look for $\mu\subseteq\lambda\subseteq\overline\mu$ where $|\lambda|=|\mu|+k$. Such $\lambda$
correspond to a choice of $k$ rows $r_1\leq r_2\leq \ldots \leq r_k$ 
to add a box, among $\ell(\mu)+k-1$ choices (rows $2,3,\ldots,\ell(\mu)+k$). There are
${\ell(\mu)+2k-2\choose k}\in |\mu|^{O(1)}$ many ways to do this. For each such choice, it takes
constant time to verify that $\lambda$ is a partition. For those cases, we construct a possible Lenart tableau
$T$ by filling row $r_i$ in at most $(r_i-1)^k$ ways. Since $r_i\leq \ell(\mu)+k$, there are $|\mu|^{O(1)}$-many possible row strictly increasing tableaux $T$. It remains to determine if $T$ is actually a Lenart tableau, which takes constant time (since $k$ is fixed). Finally, to each tableau, we must compute $f^{\lambda}$ via (\ref{eqn:HLF}). This takes
$|\lambda|^{O(1)}$-time. Now, since $|\lambda|=|\mu|+k$ and $k$ is fixed, it also takes $|\mu|^{O(1)}$-time.
Moreover, \[\log(f^{\lambda})\leq \log |\lambda|!\in O(|\mu| \log |\mu|).\] 
Hence, summing the at most 
${\ell(\mu)+2k-2\choose k}$ hook-length calculations, also takes $|\mu|^{O(1)}$-time, as desired.
\end{proof}

\begin{example}
We elaborate on the proof of Theorem~\ref{prop:plus2} in the case $k=2$. Here we look for $\mu\subseteq\lambda\subseteq\overline\mu$ where $|\lambda|=|\mu|+2$. Such $\lambda$ correspond to a choice of two rows $r_1\leq r_2$ 
to add a box, among $\ell(\mu)+1$ choices (rows $2,3,\ldots,\ell(\mu)+2$). 
If $r_1=r_2$, $g_{\mu,\lambda}={r_2-1\choose 2}$. Otherwise if $r_1<r_2$, 
there are ${\ell(\mu)+1\choose 2}$ many choices. Assuming $\lambda$ is a partition,
there are two cases. If the two boxes are in different
columns $g_{\mu,\lambda}=(r_1-1)(r_2-1)$. Otherwise, if they are in the same
column (and hence $r_2=r_1+1$), then $g_{\mu,\lambda}={r_2-1\choose 2}$. Now apply (\ref{eqn:lenart}).

For $\lambda=\delta_{100}=(100,99,\ldots,3,2,1)$ and $N={100\choose 2}+2$, this procedure
exactly computes $f^{\lambda,N}=\#{\sf Hecke}(w_0,N)= 3.75\ldots\times 10^{7981}$. \qed
\end{example}

\subsection{Application to Euler characteristics of Brill-Noether varieties (after \cite{Anderson, Chan})}
Counting standard set-valued tableaux has been given geometric impetus through work of 
\cite{Anderson, Chan} on \emph{Brill-Noether varieties}. 
More precisely, following \cite[Definition~1.2]{Chan}, let $g,r,d\in {\mathbb Z}_{\geq 0}$. Suppose $\alpha=(\alpha_0\leq \alpha_1\leq \cdots \leq \alpha_r)$ and $\beta=(\beta_0\leq \beta_1\leq \cdots \leq \beta_r)$ be sequences in ${\mathbb Z}_{\geq 0}^{r+1}$. Let ${\sf CP}={\sf CP}(g,r,d,\alpha,\beta)$ be the skew Young diagram with boxes
\[\{(x,y)\in {\mathbb Z}^2 \ : \ 0\leq y\leq r, -\alpha_y\leq x<g-d+r+\beta_{r-y}\}.\]

\begin{example}
\label{exa:Dec21}
If $g=45,d=43,r=3, \alpha=(0,1,1,4), \beta=(0,0,1,3)$ then 
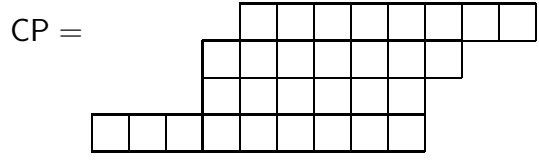
\[{\sf CP}=\tableau{&&&&{\ }&{\ }&{\ }&{ \ }&{\ }&{\ }&{\ }&{\ }\\
&&&{\ }&{\ }&{\ }&{\ }&{\ }&{\ }&{\ }\\
&&&{\ }&{\ }&{\ }&{\ }&{\ }&{\ }\\
{\ }&{\ }&{\ }&{\ }&{\ }&{\ }&{\ }&{\ }&{\ }}\]\qed
\end{example}

Let $\chi(G^{r,\alpha,\beta}_d(X,p,q))$ be the algebraic Euler characteristic of the Brill-Noether variety
$G^{r,\alpha,\beta}_d(X,p,q)$. 
\begin{theorem}[M.~Chan-N.~Pfleuger \cite{Chan}]
\label{thm:Chan}
$(-1)^{g-|{\sf CP}|}\chi(G^{r,\alpha,\beta}_d(X,p,q))=f^{{\sf CP},g}$.
\end{theorem}

Given $\lambda/\mu$, construct a ($321$-avoiding) permutation $w_{\lambda/\mu}$ by filling all boxes in the same northwest-southeast diagonal with the same entry, starting with $1$ on the northeastmost diagonal and increasing consecutively as one moves southwest. Call this filling $T_{\lambda/\mu}$.
Let $(r_1,r_2,\ldots,r_{|\lambda/\mu|})$ be the left-to-right, top-to-bottom, row reading word of $T_{\lambda/\mu}$. Define 
$w_{\lambda/\mu}=s_{r_1}s_{r_2}\cdots s_{r_{|\lambda/\mu|}}$.

\begin{example} 
\label{exa:Dec21xyz}
If $\lambda/\mu=(12,10,9,9)/(4,3,3,0)$ be ${\sf CP}$ from Example~\ref{exa:Dec21}.
Then
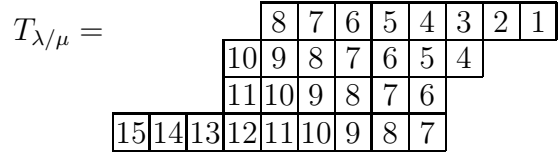
\[T_{\lambda/\mu}=\tableau{&&&& 8 & 7 & 6 & 5 & 4 &3 & 2 &1\\
&&& 10 & 9 & 8 & 7 & 6 & 5 & 4\\
&&& 11 & 10 & 9 & 8 & 7 & 6\\
15 & 14 & 13 & 12 &11 &10 & 9 & 8 & 7}\]
The reading word is
\[(8,7,6,5,4,3,2,1,10,9,8,7,6,5,4,11,10,9,8,7,6,15,14,13,12,11,10,9,8,7)\]
and
$w_{\lambda/\mu}=9,1,2,11,3,12,16,4,5,6,7,8,10,13,14,15\in {S}_{16}$.\qed
\end{example}

Earlier, in \cite{Anderson},
a \emph{cancellative} combinatorial formula for $\chi(G^{r,\alpha,\beta}_d(X,p,q))$ is given. In \cite{Chan}, another cancellative formula is given,
as a signed sum involving counts of (ordinary) skew standard Young tableaux. 
See \cite[Theorem~6.6]{Chan} (and the discussion of
\cite{Fomin.Greene} in \cite[Section~6]{Chan}) as well as \cite[Theorem~A, Theorem~C]{Anderson}.

\begin{proposition}
\label{cor:Euler}
Let $w_{\sf CP}$ is a $321$-avoiding permutation defined as above. Then:
\[(-1)^{g-|{\sf CP}|}\chi(G^{r,\alpha,\beta}_{d}(X,p,q))=\sum_{\lambda, |\lambda|=g} b_{w_{\sf CP},\lambda}f^{\lambda},\]
and
\[(-1)^{g-|{\sf CP}|}\chi(G^{r,\alpha,\beta}_{d}(X,p,q))=\sum_{\lambda: |CP|\leq |\lambda|\leq g} c_{w_{\sf CP},\lambda}f^{\lambda,g}.\]
\end{proposition}
\begin{proof}
By \cite[Section~2]{Buch:KLR},
$G_{w_{\nu/\mu}}=G_{\nu/\mu}:=\sum_{T} (-1)^{|T|-|\nu/\mu|} {\bf x}^T$,
where the sum is over semistandard set-valued tableaux of skew shape $\nu/\mu$. Therefore
if $f^{\nu/\mu,N}$ is the number of standard set-valued tableaux of this shape with $N$ entries then
\begin{equation}
\label{eqn:skewDec22}
f^{\nu/\mu,N}=\#{\sf Hecke}(w_{\nu/\mu},N).
\end{equation} 
Now combine this with Theorem~\ref{thm:Chan},
Proposition~\ref{prop:usefulChan} and~\ref{prop:secondG}.
\end{proof}

The second formula expresses $(-1)^{g-|{\sf CP}|}\chi(G^{r,\alpha,\beta}_{d}(X,p,q))$
as a cancellation-free sum of Euler characteristics of other Brill-Noether varieties. \emph{Is there a geometric explanation of this?}

\section{Three importance sampling algorithms}
\label{sec:5}

\subsection{Estimating $\#{\sf Red}(w)$} \label{subsection:5.1}

Define a random variable ${\sf Y}_w$ for $w\in {S}_n$, as follows:

\hrulefill

\indent
if $w$ is vexillary then\\
\indent \indent\indent
${\sf Y}_w=f^{\lambda(w)}$\\
\indent
else\\
\indent\indent\indent
$C=\{w' \text{\ is a child of $w$ in ${\mathcal T}(w)$}\}$\\
\indent\indent\indent
Choose $W'\in C$ uniformly at random\\
\indent\indent\indent
${\sf Y}_w=\#C \times {\sf Y}_{w'}$

\hrulefill

\begin{proposition}
\label{thm:Yalgorithm}
Let $w\in {S}_n$. Then
${\mathbb E}({\sf Y}_w)=\#{\sf Red}(w)$. 
\end{proposition}
\begin{proof}
We induct on $h=h(w)\geq 0$, the height of ${\mathcal T}(w)$, i.e., the maximum length of any path
from the root to a leaf. In the base case, $h=0$, $w$ is vexillary and thus, by (\ref{eqn:whenvex}), 
\[{\mathbb E}({\sf Y}_w)=f^{\lambda(w)}=\#{\sf Red}(w).\] 

Our induction hypothesis is that  ${\mathbb E}({\sf Y}_u)=\#{\sf Red}(u)$ whenever $h(u)<h(w)$. Now
\begin{align*}
{\mathbb E}({\sf Y}_w)= & \sum_{w'\in C}{\mathbb  E}({\sf Y}_w|W'=w'){\mathbb P}(W'=w')\\
= & \frac{1}{\#C}\sum_{w'\in C}{\mathbb E}({\sf Y}_w|W'=w')\\
= &\frac{1}{\#C}\sum_{w'\in C}{\mathbb E}(\#C\times {\sf Y}_{w'})\\
=&\sum_{w'\in C}{\mathbb E}({\sf Y}_{w'})\\
=&\sum_{w'\in C}\#{\sf Red}(w') \text{\ \ (induction hypothesis)}\\
= & \#{\sf Red}(w).
\end{align*}
The last equality is by construction of the transition algorithm and Theorem~\ref{thm:noorig}.
\end{proof}

\begin{example}
Let $w=43817625\in {S}_8$. We have the following
sequence of transition steps
\[43817625\stackrel{3}{\to} 53817426\stackrel{1}{\to} 53827146\stackrel{3}{\to} 63825147\stackrel{2}{\to} 63842157 \stackrel{2}{\to} 73642158.\]
The number of children is indicated at each stage. The final permutation is
vexillary, and 
$f^{\lambda(73642158)}=f^{6,4,2,2,1}=243243$.
Hence one sample is $3\times 1\times 3\times 2\times 2\times 243243=8756748$.
Using sample size $2\times 10^3$ gives an estimate of $2.09(\pm 0.04)\times 10^6$, versus
$\#{\sf Red}(w)=2085655$.\footnote{The ``$(\pm 0.04)$'' refers to the standard error of the mean. All estimates are based on twelve trials of an indicated sample size. Code is available at \url{https://github.com/ICLUE/reduced-word-enumeration}}\qed
\end{example}

\begin{example}[$w=\sigma^{(n)}=2143\cdots 2n \ 2n-1$]
When $n=10$ (so $\sigma^{(n)}\in {S}_{20}$), using sample size $10^5$ gives an estimate of $3.63(\pm 0.02)\times 10^6$, 
which is close to the exact value $10!=3628800$. 
When $n=30$ ($\sigma^{(n)}\in {S}_{60}$), using sample size $2\times 10^6$ one estimates
$2.18(\pm 0.49)\times 10^{32}$ whereas $30!= 2.65\ldots \times 10^{32}$. 
\qed
\end{example}

\begin{example}[Estimating the number of skew standard Young tableaux]
\label{exa:skewSYT}
We continue Example~\ref{exa:Dec21xyz}.
Let $f^{\lambda/\mu}$ be the number of standard Young tableaux of shape $\lambda/\mu$.
By a result of S.~Billey-W.~Jockusch-R.~P.~Stanley \cite[Corollary~2.4]{BJS},
$F_{w_{\lambda/\mu}}=s_{\lambda/\mu}$.
Taking the coefficient of $x_1 x_2\cdots x_{|\lambda/\mu|}$ on both sides implies
$\#{\sf Red}(w_{\lambda/\mu})=f^{\lambda/\mu}$. One has the textbook determinantal formula 
\begin{equation}
\label{eqn:thedet}
f^{\lambda/\mu}=\left|\lambda/\mu \right|!\det\left(\frac{1}{(\lambda_i-\mu_j-i+j)!}\right)_{i,j=1}^t.
\end{equation}
So $f^{\lambda/\mu}\!=\!73064598262110\!\approx\! 7.31\times 10^{13}$.
A $10^4$ sample size estimate is $7.30(\pm 0.04)\times 10^{13}$.
\qed
\end{example}

\subsection{Estimating $\#{\sf Hecke}(w,N)$}
We propose a different
importance sampling algorithm, to compute $\#{\sf Hecke}(w,N)$. For $N<\ell(w)$ 
the random variable ${\sf Z}_{w,N}$ is equal to $0$ and for $N\geq \ell(w)$, it is recursively defined by:

\hrulefill

\indent
if $w=id$  then\\
\indent \indent\indent
if $N=0$ then
\ ${\sf Z}_{w,N}=1$ \ else \
${\sf Z}_{w,N}=0$\\
\indent
else\\
\indent\indent\indent
$D=\{i:w(i)>w(i+1)\}$\\
\indent\indent\indent
Choose $I\in D$ and $\theta\in\{0,1\}$ independently and uniformly at random\\
\indent\indent\indent
if $\theta=0$ then
 \ ${\sf Z}_{w,N}=2\#D\times {\sf Z}_{w,N-1}$ \ 
else
 \ ${\sf Z}_{w,N}=2\#D\times {\sf Z}_{ws_I,N-1}$

\hrulefill

\begin{proposition}
\label{prop:Heckealg123}
Let $w\in {S}_n$ and $N\geq \ell(w)$. Then
${\mathbb E}({\sf Z}_{w,N})=\#{\sf Hecke}(w,N)$.
\end{proposition}
\begin{proof}
First we claim
\begin{equation}
\label{eqn:Heckething}
\#{\sf Hecke}(w,N)=
\begin{cases}
1 & \text{if $w=id$ and $N=0$}\\
0 & \text{if $w=id$ and $N>0$}\\
\sum_{i\in D}(\#{\sf Hecke}(ws_i,N-1)+\#{\sf Hecke}(w,N-1)) & \text{otherwise.}\\
\end{cases}
\end{equation}
The unique Hecke word for $w=id$ is the empty word; this explains the first two cases. 

Thus assume $w\neq id$ and $N\geq \ell(w)$. Suppose that
$(i_1, i_2, \ldots, i_N)\in {\sf Hecke}(w,N)$.
\begin{claim}
\label{claim:blah123}
$i_N$ is the position of a descent of $w$, i.e., $w(i_N)>w(i_N +1)$.
\end{claim}
\noindent
\emph{Proof of Claim~\ref{claim:blah123}:}
Consider $w':=s_{i_1}\star s_{i_2}\star \cdots \star s_{i_{N-1}}$.
Either $\ell(w')=\ell(w)$ or $\ell(w')=\ell(w)-1$. In the former case
then if $i_N$ is the position of an ascent of $w'=w$ then $w=w'\star s_{i_N}$ would create a descent at that position, a contradiction. In the latter case, $w'$ had an ascent at position $i_N$ which becomes a descent in $w'\star s_{i_N}=w's_{i_N}$.
\qed

Claim~\ref{claim:blah123} implies the existence of a bijection
\begin{equation}
\label{eqn:vcvmay23}
{\sf Hecke}(w,N)\stackrel{\sim}{\to} \left(\bigcup_{i\in D} {\sf Hecke}(ws_i,N-1)\times \{i\}\right)\cup \left(\bigcup_{i\in D} {\sf Hecke}(w,N-1)\times \{i\}\right),
\end{equation}
defined by $(i_1,i_2,\ldots, i_{N-1}, i_N)\in {\sf Hecke}(w,N)\mapsto ((i_1,i_2,\ldots,i_{N-1}),i_N)$.\footnote{
If $N=\ell(w)$, then ${\sf Hecke}(w,N)={\sf Red}(w)$ and ${\sf Hecke}(w,N-1)=\emptyset$. In this case, (\ref{eqn:vcvmay23}) reduces to the bijection
${\sf Red}(w)\stackrel{\sim}{\rightarrow} \bigcup_{i\in D} {\sf Red}(ws_i)\times\{i\}$.}
Therefore, by taking cardinalities on both sides of (\ref{eqn:vcvmay23}) we obtain the third case of (\ref{eqn:Heckething}).

\excise{
Suppose $N=\ell(w)$. We induct on $\ell(w)\geq 0$. The case $\ell(w)=0$
holds by the first case of (\ref{eqn:Heckething}). Next, if $\ell(w)=1$ then $w=s_i$ for some $i$.
Assume $\ell(w)>1$ and ${\mathbb E}({\sf Z}_{u,\ell(u)})=\#{\sf Red}(u)$ for all $u$ with $\ell(u)<\ell(w)$.
 \begin{align*}
{\mathbb E}({\sf Z}_{w,\ell(w)})= & \sum_{i\in D}{\mathbb  E}({\sf Z}_{w,\ell(w)}|I=i){\mathbb P}(I=i)\\
= & \frac{1}{\#D}\sum_{i\in D}{\mathbb E}({\sf Z}_{w,\ell(w)}|I=i)\\
= &\frac{1}{\#D}\sum_{i\in D}{\mathbb E}(\#D\times {\sf Z}_{ws_i,\ell(w)-1})
\end{align*}
\begin{align*}
=&\sum_{i\in D}{\mathbb E}({\sf Z}_{ws_i,\ell(ws_i)}) \text{\ \ ($\ell(w)-1=\ell(ws_i)$ since $i\in D$)}\\
=&\sum_{i\in D}\#{\sf Red}(ws_i) \text{\ \ (induction hypothesis)}\\
= & \#{\sf Red}(w) \text{\ \  \ \ \ \ \ \ \ \ \ \ (by the third case of (\ref{eqn:Heckething})),}
\end{align*}
as required. 

Now suppose $N>\ell(w)$.}

Returning to proposition itself, we induct on $N\geq 0$. The case $N=0$ holds by the first case of
(\ref{eqn:Heckething}) and the definition $Z_{w,N}=0$ if $N<\ell(w)$. 
For $N>0$,
\begin{align*}
{\mathbb E}({\sf Z}_{w,N}) & = \sum_{i \in D}{\mathbb E}({\sf Z}_{w,N}|I=i,\theta=0){\mathbb P}(I=i){\mathbb P}(\theta=0)\\
& \ +\sum_{i \in D}{\mathbb E}({\sf Z}_{w,N}|I=i,\theta=1){\mathbb P}(I=i){\mathbb P}(\theta=1)\\
&=\sum_{i\in D}{\mathbb E}(2\#D\times {\sf Z}_{w,N-1})\frac{1}{\#D}\times \frac{1}{2}+ \sum_{i\in D}{\mathbb E}(2\#D\times {\sf Z}_{ws_i,N-1})\frac{1}{\#D}\times \frac{1}{2}\\
& = \sum_{i\in D}\left({\mathbb E}({\sf Z}_{w,N-1})+{\mathbb E}({\sf Z}_{ws_i,N-1})\right)\\
& = \sum_{i\in D}\left(\#{\sf Hecke}(w,N-1)+\#{\sf Hecke}(ws_i,N-1)\right)\\
& = \#{\sf Hecke}(w,N),
\end{align*}
where we have applied induction (on $N$) and the third case of (\ref{eqn:Heckething}). 
\end{proof}

\begin{example}
\label{exa:Dec26a}
One can explicitly generate all $2030964$ elements of ${\sf Hecke}(351624, 13)$.
 A $2000$ sample size estimate is $2.04(\pm 0.10)\times 10^6$. \qed 
\end{example}
	
\begin{example}
\label{exa:bad}
By \cite[Corollary~1.3]{ReinerTennerYong},
\begin{equation}
\label{eqn:RTYplus2}
\#{\sf Hecke}\left(w_0,{n\choose 2}+1\right)=\frac{{n\choose 2}\left[{n\choose 2}+1\right]}{n}\times \#{\sf Red}(w_0).
\end{equation}

For $n=10$, $\#{\sf Hecke}(w_0,46)=5.65\ldots \times 10^{28}$.
Using sample size $10^8$, we obtained an estimate of $\approx 4.26(\pm 1.94)\times 10^{28}$. 
\qed
\end{example}

The ${\sf Z}$-algorithm restricts to an algorithm to compute $\#{\sf Red}(w)$. However, the ${\sf Y}$-algorithm of Subsection~\ref{subsection:5.1} sometimes has better convergence in this case. This suggests a 
``hybrid'' algorithm. Define ${\sf H}_{w,N}$ to be $0$ if $N<\ell(w)$. Otherwise,

\hrulefill

\indent
if $N=\ell(w)$ then ${\sf H}_{w,N}=Y_w$\\
\indent
else if $w=id$  then\\
\indent \indent\indent
if $N=0$ then
\ ${\sf H}_{w,N}=1$ \ else\ 
${\sf H}_{w,N}=0$\\
\indent
else\\
\indent\indent\indent
$D=\{i:w(i)>w(i+1)\}$\\
\indent\indent\indent
Choose $I\in D$ and $\theta\in\{0,1\}$ independently and uniformly at random\\
\indent\indent\indent
if $\theta=0$ then
  \ ${\sf H}_{w,N}=2\#D\times {\sf H}_{w,N-1}$ \ 
else
  \ ${\sf H}_{w,N}=2\#D\times {\sf H}_{ws_I,N-1}$

\hrulefill

\begin{proposition}
Let $w\in S_n$. Then ${\mathbb E}[{\sf H}_{w,N}]=\#{\sf Hecke}(w,N)$.
\end{proposition}

We omit the proof, as it is a straightforward modification of the argument for Proposition~\ref{prop:Heckealg123}, using Proposition~\ref{thm:Yalgorithm}.

\begin{example}
Let $w=361824795\in S_9$; hence $\ell(w)=12$. Using sample size $10^6$ with the ${\sf Z}$ algorithm
gives $\#{\sf Hecke}(w,25)\approx
5.98(\pm 0.04)\times 10^{16}$.
The estimate from the ${\sf H}$ algorithm (with the same sample size) is $\#{\sf Hecke}(w,25)\approx
6.02(\pm 0.08)\times 10^{16}$.
For Example~\ref{exa:bad}, with $10^8$ samples, the {\sf H} algorithm estimates 
 $\#{\sf Hecke}(w_0,46)$
as $6.09(\pm 4.69)\times 10^{28}$. 
\qed
\end{example}

\begin{example} We use Proposition~\ref{prop:plus2} to compute $\#{\sf Hecke}(w_0,{n\choose 2}+2)$.
When $n=7$, $\#{\sf Hecke}(w_0,23)=2.54\ldots\times 10^{12}$.
A $10^6$ sample size estimate is $2.60(\pm 0.22)\times 10^{12}$.  
For $n=10$, $\#{\sf Hecke}(w_0, 47)=6.01\ldots\times 10^{30}$. A $10^8$ sample size estimate is $\approx 4.04(\pm 2.17)\times 10^{30}$.
\qed
\end{example}

\begin{example}[Skew set-valued tableaux]
\label{exa:Dec26b}
To estimate $f^{\lambda/\mu,N}$ for  $\lambda/\mu=(12,10,9,9)/(4,3,3,0)$ and $N=45$, 
we use (\ref{eqn:skewDec22}) and the ${\sf Z}$-algorithm with sample size $10^7$ 
to predict $f^{\lambda/\mu,45}=\#{\sf Hecke}(w_{\lambda/\mu},45)\approx 1.30(\pm 0.03)\times 10^{33}$.
This is backed by the estimate $1.29(\pm 0.06)\times 10^{33}$ using the ${\sf H}$-algorithm with sample size $10^6$. 
We have thus estimated the value of $(-1)^{g-|{\sf CP}|}\chi(G^{r,\alpha,\beta}_d(X,p,q))$ for the
parameters of Example~\ref{exa:Dec21}. There are a number of ways to theoretically compute this value
(\cite{Anderson}, \cite{Chan}, Proposition~\ref{cor:Euler}).
\emph{What is the exact value?} \qed
\end{example}

\section*{Acknowledgments}
We thank Anshul Adve, David Anderson, Alexander Barvinok, Melody Chan, Yuguo Chen, Anna Chlopecki, Michael Engen, Neil Fan, Sergey Fomin, Sam Hopkins, Allen Knutson, Tejo Nutalapati, Gidon Orelowitz, Colleen Robichaux, Renming Song, John Stembridge, Anna Weigandt and Harshit Yadav for helpful remarks/discussion.
We are especially grateful to Brendan Pawlowski for pointing out Theorem~\ref{thm:averagemain} appears as Theorem~3.2.7 of \cite{Brendan:thesis}, as well 
as other remarks.
AY was supported by an NSF grant and a Simons Collaboration Grant. This work is part of
ICLUE, the Illinois Combinatorics Lab for Undergraduate Experience.

\end{document}